\documentclass[a4paper]{article}

\usepackage[inner=30mm,outer=30mm,textheight=225mm]{geometry}

\usepackage{amsmath}      % For mathematical symbols
\usepackage{amssymb}      % For mathematical symbols
\usepackage{enumerate}    % To support \begin{enumerate}[(i)]
\usepackage{graphicx}     % For embedding figures
\usepackage{theorem}      % For custom theorem styles
\usepackage{url}          % For embedding URLs nicely in text

\newtheorem{theorem}{Theorem}[section]
\newtheorem{defn}[theorem]{Definition}
\newtheorem{lemma}[theorem]{Lemma}
\newtheorem{algorithm}[theorem]{Algorithm}

\newcommand{\prooflabel}{Proof}
\newcommand{\qed}{\hfill\rule[-0.5mm]{1.5mm}{3.0mm}}
\newtheorem{proofthm}{\prooflabel}
\newenvironment{proof}{\begin{proofthm} \em}{\qed \end{proofthm}}

\newcommand{\cpp}{C\nolinebreak\hspace{-.05em}\raisebox{.4ex}{\tiny\bf
    +}\nolinebreak\hspace{-.10em}\raisebox{.4ex}{\tiny\bf +}}
\newcommand{\ccpp}{C/\cpp}
\newcommand{\fxrays}{{\em FXrays}}
\newcommand{\mgap}{\phantom{-}}
\newcommand{\normalsw}{{\em Normal}}
\newcommand{\R}{\mathbb{R}}
\newcommand{\rank}[1]{\mathrm{rank}(#1)}
\newcommand{\regina}{{\em Regina}}

\newcommand{\trialplots}[1]{
    \centerline{\begin{tabular}{cc}
        \includegraphics[width=6.2cm]{#1-shyp.eps} &
        \includegraphics[width=6.2cm]{#1-qhyp.eps} \smallskip \\
        \includegraphics[width=6.2cm]{#1-stwist.eps} &
        \includegraphics[width=6.2cm]{#1-qtwist.eps}
    \end{tabular}}
}

\title{Optimizing the Double Description Method for \\
    Normal Surface Enumeration}
\author{Benjamin A.~Burton}

\date{\today}

\begin{document}

\maketitle

\abstract{Many key algorithms in 3-manifold topology involve the
    enumeration of normal surfaces, which is based upon the double
    description method for finding the vertices of a convex polytope.
    Typically we are only interested in a small subset of these vertices,
    thus opening the way for substantial optimization.
    Here we give an account of the vertex enumeration problem as it
    applies to normal surfaces, and present new optimizations that yield
    strong improvements in both running time and memory consumption.
    The resulting algorithms are tested using the freely available
    software package {\regina}.}


\section{Introduction} \label{s-intro}

Some of the most fundamental problems in 3-manifold topology are
algorithmic, such as determining the structure of a given space,
or deciding whether two spaces are topologically equivalent.
Much progress has been made on these problems;
notable examples include the unknot recognition algorithm of
Haken \cite{haken61-knot}, the 3-sphere recognition algorithm of
Rubinstein and Thompson
\cite{rubinstein95-3sphere,rubinstein97-3sphere,thompson94-thinposition},
the connected sum and JSJ decomposition algorithms of Jaco and Tollefson
\cite{jaco95-algorithms-decomposition}, and the solution to the homeomorphism
problem for Haken manifolds, developed by Haken \cite{haken62-homeomorphism}
and completed by Jaco and Oertel \cite{jaco84-haken} and
Hemion \cite{hemion92}.

Several recurring themes are found in these and many other topological
algorithms:
(i) they are extremely slow, (ii) they are extremely difficult to implement,
and (iii) they are all based on {\em normal surface theory}.

The reason normal surface theory is so prevalent is because it allows
topological existence problems to be converted into vertex enumeration
problems on polytopes, which (being numerical and discrete) are far simpler
to work with algorithmically.

Unfortunately this very technique that makes these problems approachable
also makes the resulting algorithms impractically slow for all but the
smallest 3-manifolds.  This is because vertex enumeration can grow
exponentially slow in the dimension of the polytope \cite{dyer83-complexity},
which equates to exponentially slow in the complexity of the 3-manifold.

Any practical implementation therefore requires a
highly optimized vertex enumeration algorithm.  Vertex enumeration
algorithms fall into two broad categories: those based on the double
description method of Motzkin et al.~\cite{motzkin53-dd}, and those based
on pivoting, as described for example by Dyer \cite{dyer83-complexity}.
Both classes of algorithm have been analyzed and optimized in the literature;
see for instance the optimized double description methods of Fukuda
and Prodon \cite{fukuda96-doubledesc}, or the recent
lexicographic pivoting method of Avis \cite{avis00-revised}.

If we restrict our focus to topological problems, there are further
gains to be had.  Essentially, we can exploit the fact that normal
surface algorithms typically require only a small number of polytope
vertices, namely those that correspond to {\em embedded} surfaces in the
underlying 3-manifold.  We therefore have permission to ignore ``most''
of the vertices of the polytope, which opens the door to substantial
improvements in efficiency.

The purpose of this paper is twofold.  First we give a full description of
the standard normal surface enumeration algorithm, which combines the double
description method with the filtering method of Letscher---although well
known, there is no account of this algorithm in the present literature.
We then improve this algorithm by combining techniques from the
literature with new ideas, cutting both running time and memory usage by
orders of magnitude as a result.

We focus only on the double description method in this paper.  Pivoting
algorithms are certainly appealing, particularly because of their
extremely low memory footprint \cite{avis00-revised,avis92-pivot}.
However, it is difficult to exploit the embedded surface constraints
with these algorithms.  We discuss this in more detail in Section~\ref{s-dd}.

On a practical note, there are two well-known implementations for the
enumeration of normal surfaces: {\fxrays} \cite{fxrays}, by Culler and
Dunfield, and {\regina} \cite{regina,burton04-regina}, by the author.
Both are freely available under the GNU General Public License.
David Letscher wrote a proof-of-concept program {\normalsw} in 1999 that
preceded both implementations, but his software is no longer available.

Each implementation has different design goals.  {\fxrays} uses highly
streamlined code and data structures, and is very fast for the problems
that it is designed for.  {\regina} on the other hand is more failsafe and
applicable to a wider range of problems, but pays a penalty in both time
and memory usage.  As an example,
{\fxrays} uses native integers where {\regina} uses
arbitrary precision arithmetic, which makes {\fxrays} faster and smaller
but also at risk of integer overflow (which it detects but
cannot overcome).  {\regina} also uses slower filtering methods, but
these generalize well to the sister problem of {\em almost normal surface}
enumeration, which appears in some of the high-level topological
algorithms mentioned earlier.

Since our concern here is the underlying algorithms, we focus on a
single implementation (in this case, {\regina}).
All of the improvements described here have been built
into {\regina} version~4.5.1,
released in October 2008, and it is pleasing to see that this new code
enjoys significantly better time and memory performance than has been
seen in either software package in the past.

The remainder of this paper is structured as follows.
Section~\ref{s-normal} begins with an outline of normal surface theory,
focusing on its connections to polytope vertex enumeration.  In
Section~\ref{s-dd} we describe the double description method, and
explain how the filtering method of Letscher allows us to concentrate only
on embedded surfaces.  Section~\ref{s-opt} presents a series of
implementation techniques and algorithmic improvements that further
optimize these core algorithms.
These optimizations are put to the test in Section~\ref{s-expt}
with experimental measurements of running time and
memory consumption, and
Section~\ref{s-conc} concludes with a summary of our findings.

The author is indebted to Bernard Blackham for his helpful suggestions
regarding micro-optimization, and for highlighting the excellent
references \cite{drepper07-memory,warren02-hackers-delight} on this topic.
Thanks must also go to the University of Melbourne for their continued
support of the development of {\regina}.

\section{Normal Surface Theory} \label{s-normal}

Normal surfaces were introduced by Kneser \cite{kneser29-normal}, and
further developed by Haken \cite{haken61-knot,haken62-homeomorphism} for
use with the unknot recognition problem and the homeomorphism problem.
They are now commonplace in recognition and decomposition algorithms,
and more recently have found applications in simplification algorithms
\cite{jaco03-0-efficiency}.

From a practical perspective, many of these
algorithms are extremely messy and difficult to implement, due to the
complex geometric operations involved and the myriad of problematic cases.
Some have only recently been implemented in practice, such as
the 3-sphere recognition and connected sum decomposition algorithms
in {\regina}; others, such as JSJ decomposition
or Haken's homeomorphism algorithm, have never been
implemented at all.  Recent techniques have been developed to reduce
both the difficulty and inefficiency of these algorithms; examples
include Tollefson's quadrilateral space \cite{tollefson98-quadspace},
the crushing method of Jaco and Rubinstein \cite{jaco03-0-efficiency},
and the ``guts'' analysis of Jaco et al.~\cite{jaco02-algorithms-essential}.

Since the focus of this paper is on the double description method, we offer
very little topological background, concentrating instead on the linear
programming aspects of normal surface theory.  For a more extensive
review of normal surfaces, the reader is referred to \cite{hass99-knotnp}
or \cite{hemion92}.

\subsection{Triangulations and Normal Surfaces}

The key topological structures that we work with in this paper are
{\em triangulations} and {\em normal surfaces}.  We proceed to define
each of these in turn.

Triangulations are representations of 3-manifolds that are ideal for
computation.  They are discrete structures, and they are very
general---it is usually a simple matter to convert some other
description of a 3-manifold (such as a Heegaard splitting or
a Dehn filling) into a triangulation, whereas the
other direction is often more difficult.
Each of the high-level topological algorithms listed in the introduction
takes a 3-manifold triangulation as input.

\begin{defn}[Triangulation]
    A {\em 3-manifold triangulation of size $n$} is a finite collection of
    $n$ tetrahedra, where some or all of the $4n$ tetrahedron faces are
    affinely identified in pairs, and where the resulting topological space
    forms a 3-manifold.  We allow identifications between different faces
    of the same tetrahedron, and likewise with edges and vertices.  By a
    {\em face, edge or vertex of the triangulation}, we refer to an
    equivalence class of tetrahedron faces, edges or vertices
    under these identifications.
\end{defn}

\begin{figure}[htb]
\centerline{\includegraphics[scale=0.6]{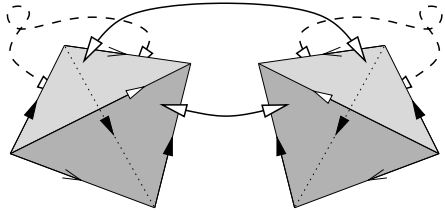}}
\caption{A two-tetrahedron triangulation of $S^2 \times S^1$}
\label{fig-s2xs1}
\end{figure}

As an example, Figure~\ref{fig-s2xs1} shows a size two triangulation of the
product space $S^2 \times S^1$.  In each tetrahedron
the two rear faces are identified with a twist, and the two front faces
of one tetrahedron are identified with the two front faces of the other.
The triangulation only has one vertex (since all eight tetrahedron
vertices are identified) and it has three distinct edges, indicated
in the diagram by three different styles of arrowhead.

Normal surfaces are two-dimensional surfaces within a triangulation that
intersect individual tetrahedra in a well-controlled fashion.
These well-controlled intersections, defined in terms of normal discs,
make it easier to analyze and
search for surfaces that tell us about the topology of the
underlying manifold.

\begin{defn}[Normal Disc]
    Let $\Delta$ be a tetrahedron in a 3-manifold triangulation.
    A {\em normal disc} in $\Delta$ is a topological disc embedded in
    $\Delta$ which does not touch any vertices of $\Delta$, whose
    interior lies in the interior of $\Delta$, and whose boundary
    consists of either (i) three arcs, running across the three faces
    surrounding some vertex, or (ii) four arcs, running across all four
    faces of the tetrahedron.  Discs with three or four boundary arcs
    are called {\em triangles} or {\em quadrilaterals} respectively.
    Several normal discs are illustrated in Figure~\ref{fig-normaldiscs}.

    \begin{figure}[htb]
    \centerline{\includegraphics[scale=0.7]{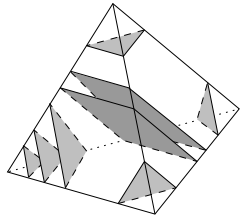}}
    \caption{Normal discs inside a tetrahedron}
    \label{fig-normaldiscs}
    \end{figure}

    There are seven different {\em types} of normal disc within a
    tetrahedron, defined by the choice of tetrahedron edges that a disc
    intersects.  In particular, there are four triangle types (each
    meeting three edges) and three quadrilateral types (each meeting
    four edges), as illustrated in Figure~\ref{fig-normaltypes}.

    \begin{figure}[htb]
    \centerline{\includegraphics[scale=0.5]{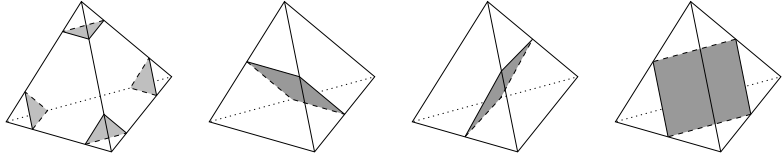}}
    \caption{The seven different types of normal disc in a tetrahedron}
    \label{fig-normaltypes}
    \end{figure}
\end{defn}

\begin{defn}[Normal Surface]
    Let $\mathcal{T}$ be a 3-manifold triangulation.  An {\em embedded normal
    surface} in $\mathcal{T}$ is a properly embedded surface in $\mathcal{T}$
    that meets each tetrahedron in a collection of disjoint normal discs.
    Here we also allow disconnected surfaces (i.e., disjoint unions of
    smaller surfaces).
\end{defn}

\begin{figure}[htb]
\centerline{\includegraphics[scale=0.6]{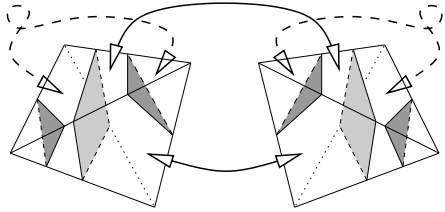}}
\caption{An embedded normal surface inside the triangulation
    of $S^2 \times S^1$}
\label{fig-s2xs1-normal}
\end{figure}

To illustrate, Figure~\ref{fig-s2xs1-normal} shows an embedded
normal surface inside the triangulation of $S^2 \times S^1$ that was
discussed earlier.
The identifications of tetrahedron faces cause the six normal discs to join
together to form a 2-sphere (which turns out to be a 2-sphere
at one ``level'' of the product $S^2 \times S^1$).

\subsection{The Projective Solution Space}

A key strength of normal surfaces is their ability to bridge the worlds
of 3-manifold topology and linear algebra.  We do this through the
vector representation of a normal surface, which is defined below.

Throughout this section, we assume that $\mathcal{T}$ is a triangulation
with $n$ tetrahedra, labelled $\Delta_1,\ldots,\Delta_n$.  For each
tetrahedron, we arbitrarily number its triangular normal disc types $1,2,3,4$
and its quadrilateral normal disc types $1,2,3$.

\begin{defn}[Vector Representation] \label{d-vecrep}
    Let $S$ be an embedded normal surface within triangulation $\mathcal{T}$.
    For each tetrahedron $\Delta_i$, let
    $t_{i,j}$ be the number of triangular discs of the $j$th type
    contained in $S$ ($j=1,2,3,4$), and let $q_{i,k}$ be the number of
    quadrilateral discs of the $k$th type contained in $S$ ($k=1,2,3$).
    Then the {\em vector representation} of the surface $S$ is the
    $7n$-dimensional vector
    \[ (
        \ t_{1,1},t_{1,2},t_{1,3},t_{1,4},\ q_{1,1},q_{1,2},q_{1,3}\ ;
        \ t_{2,1},t_{2,2},t_{2,3},t_{2,4},\ q_{2,1},q_{2,2},q_{2,3}\ ;
        \ \ldots,q_{n,3}\ ). \]
\end{defn}

Essentially the vector representation merely counts the number of normal
discs of each type in each tetrahedron.
As shown by Haken \cite{haken61-knot}, this gives enough information to
uniquely identify the surface, since there is only one way to glue the
normal discs together without causing the surface to intersect itself:

\begin{lemma}
    Let $S_1$ and $S_2$ be embedded normal surfaces in triangulation
    $\mathcal{T}$.  If the vector representations of $S_1$ and $S_2$ are
    identical, then the surfaces $S_1$ and $S_2$ are ambient isotopic.
\end{lemma}

While every normal surface has a vector representation, there are of course
$7n$-dimen\-sional vectors that do not correspond to any embedded normal
surface at all.  It is therefore useful to pin down necessary and
sufficient conditions on the vector representation.

\begin{defn}[Admissible Vector] \label{d-admissible}
    Let $\mathbf{v} = ( \ t_{1,1},t_{1,2},t_{1,3},t_{1,4},
        \ q_{1,1},q_{1,2},q_{1,3}\ ; \ \ldots,q_{n,3}\ )$
    be a $7n$-dimensional vector of integers.
    This vector is {\em admissible} if it satisfies the following constraints:
    \begin{itemize}
        \item {\em Non-negativity:} Every coordinate of $\mathbf{v}$ is
        non-negative.
        \item {\em Matching equations:} Consider any two tetrahedron
        faces that are identified in the triangulation; suppose that the
        relevant tetrahedra are $\Delta_i$ and $\Delta_j$ (where $i=j$ is
        allowed).  Let $F$ denote the resulting face of the triangulation,
        and let $e$ be any one of the three edges surrounding face $F$.
        We obtain an equation from $F$ and $e$ as follows.

        Precisely one of the four triangle types and one of the three
        quadrilateral types in each of $\Delta_i$ and $\Delta_j$ meets
        $F$ in arcs parallel to $e$.  Let $t_{i,a}$, $q_{i,b}$, $t_{j,c}$
        and $q_{j,d}$ be the corresponding coordinates of $\mathbf{v}$.
        Then it is true that $t_{i,a} + q_{i,b} = t_{j,c} + q_{j,d}$.

        \item {\em Quadrilateral constraints:} For each $i \in \{1,\ldots,n\}$,
        at most one of the coordinates $q_{i,1}$, $q_{i,2}$ and $q_{i,3}$ is
        non-zero.
    \end{itemize}
\end{defn}

It is straightforward to see that the vector representation of any
embedded normal surface must be admissible:
\begin{itemize}
    \item Non-negativity is clear because the vector representation
    counts discs.
    \item The matching equations express the fact that we must be
    able to glue together discs from adjacent tetrahedra.
    This is illustrated in Figure~\ref{fig-normaldiscsadj}, where we see
    one triangle and one quadrilateral from $\Delta_i$ meeting two
    triangles from $\Delta_j$.  The corresponding matching equation,
    derived from face $F$ and edge $e$,
    states that $t_{i,a} + q_{i,b}$ ($1 + 1$) equals
    $t_{j,c} + q_{j,d}$ ($2 + 0$).

    \begin{figure}[htb]
    \centerline{\includegraphics{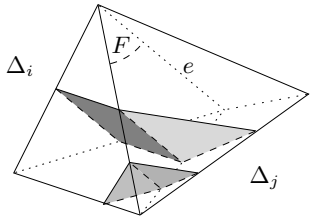}}
    \caption{The matching equations at work}
    \label{fig-normaldiscsadj}
    \end{figure}

    \item The quadrilateral constraints arise because any two quadrilaterals
    of different types in the same tetrahedron must intersect; this
    would make the resulting surface non-embedded.
\end{itemize}
A more interesting result of Haken \cite{haken61-knot} is that this
implication works both ways:

\begin{theorem} \label{t-admissible}
    Let $\mathbf{v}$ be a $7n$-dimensional integer vector that is
    not the zero vector.  Then $\mathbf{v}$ is the vector representation
    of an embedded normal surface in the triangulation $\mathcal{T}$
    if and only if $\mathbf{v}$ is admissible.
\end{theorem}

It follows that, if we can characterize
the non-negative solutions to the matching equations and the quadrilateral
constraints, then we can completely characterize the space of embedded
normal surfaces.
That is, we can effectively convert topological questions into {\em algebraic}
questions, granting us access to a wealth of knowledge in linear algebra
and linear programming.

\begin{defn}[Projective Solution Space] \label{d-proj}
    Let $N \subseteq \R^{7n}$ be the set of vectors whose coordinates
    are non-negative and which satisfy the matching equations of
    Definition~\ref{d-admissible}.  Since the matching equations define
    a linear subspace of $\R^{7n}$, it follows that $N$ is a
    convex polyhedral cone with the origin as its vertex.

    Let $P \subseteq N$ be the set of vectors in $N$ whose coordinates
    sum to one, that is, the intersection of the cone $N$ with the
    hyperplane $\sum t_{i,j} + \sum q_{i,k} = 1$.  Then $P$ is a (bounded)
    convex polytope in $\R^{7n}$, and is called the {\em projective
    solution space} for the original triangulation $\mathcal{T}$.
\end{defn}

The importance of the projective solution space comes from
the following observation:
For many definitions of ``interesting'', if a 3-manifold contains an
interesting surface, then (with some rescaling) such a surface must
appear {\em at a vertex of the projective solution space}.
For instance, this is true of
essential discs and spheres \cite{jaco95-algorithms-decomposition}, and of
two-sided incompressible surfaces \cite{jaco84-haken}.
This immediately yields an algorithm for testing whether an ``interesting''
surface exists:
\begin{itemize}
    \item Enumerate the (finitely many) vertices of the projective
    solution space;
    \item For each vertex that satisfies the quadrilateral constraints,
    reconstruct the corresponding normal surface\footnote{Of course
    many integer vectors scale down to the same vertex; however, we can
    usually restrict our attention to the smallest such vector and
    possibly its double cover.} and test whether it is interesting.
\end{itemize}
This fundamental process sits at the core of every high-level
topological algorithm described in the introduction, and many more
besides.

As an implementation note, Tollefson \cite{tollefson98-quadspace} shows that
in many scenarios the much smaller space $\R^{3n}$ can be used, where we
define vectors that have only quadrilateral coordinates.
The matching equations look
different, but the overall procedure is much the same.  The results
presented in this paper apply equally well to both Tollefson's
quadrilateral space and the
standard framework described above, and so we direct the reader to
papers such as \cite{jaco03-0-efficiency} or \cite{tollefson98-quadspace}
for further details.

\section{The Double Description Method} \label{s-dd}

As described in the previous section, many high-level topological
algorithms have at their core a polytope vertex enumeration procedure.
Specifically, we must enumerate the vertices of the {\em projective solution
space}.  If we let $d$ denote the dimension of the surrounding vector space
(so $d=7n$ in the framework of the previous section, or $3n$ in Tollefson's
quadrilateral space), then the projective solution space is a convex
polytope formed by the intersection of:
\begin{itemize}
    \item the {\em non-negative orthant} in $\R^d$, defined as
    $O = \{ \mathbf{x} \in \R^d\,|\,x_i \geq 0\ \mbox{for all $i$}\}$;
    \item the {\em projective hyperplane}, defined as
    $J = \{ \mathbf{x} \in \R^d\,|\,\sum x_i = 1\}$;
    \item the {\em matching hyperplanes} $H_1,\ldots,H_g$, where each
    hyperplane $H_i$ contains all solutions to the $i$th matching equation.
    We write the $i$th matching equation as
    $\mathbf{m}^{(i)} \cdot \mathbf{x} = 0$ for some coefficient vector
    $\mathbf{m}^{(i)} \in \R^d$, and we assume there are $g$ matching equations
    in total.
\end{itemize}

Here we have replaced the triangle and quadrilateral coordinates
$t_{i,j}$ and $q_{i,k}$ of the previous section with generic
coordinates $\mathbf{x} = (x_1,\ldots,x_d)$.  This becomes
convenient from here onwards, reflecting the fact
that we have stepped out of the world of topology and into the world
of linear programming.

The one glaring omission from the list above is the quadrilateral
constraints.  They do not feature in the definition of the
projective solution space because they break desirable properties
such as convexity and even connectedness.  Nonetheless, they play a
critical role in the enumeration algorithm; we return to this shortly.

This section is structured as follows.  We begin in
Section~\ref{s-dd-simple} with an overview of the classical double
description method as it applies to the projective solution space.
In Section~\ref{s-dd-filter} we incorporate the quadrilateral
constraints using the filtering method of Letscher, and
Section~\ref{s-dd-complexity} follows with a discussion of how bad the
performance can become
and why we do not consider alternative pivoting methods instead.

\subsection{A Simple Implementation} \label{s-dd-simple}

The {\em double description method}, devised by
Motzkin et al.~\cite{motzkin53-dd} in the 1950s and refined by many
authors since, is an incremental vertex enumeration algorithm.  The
input is a polytope described as an intersection of half-spaces and/or
hyperplanes, and the output is this same polytope described as the
convex hull of its vertices.  It takes on many guises; the flavour we
describe here is the one most convenient for the problem at hand.

\begin{algorithm}[Double Description Method] \label{a-dd}
    Recall that the projective solution space is defined to be the
    intersection $P = O \cap J \cap H_1 \cap H_2 \cap \ldots \cap H_g$,
    where $O$ is the non-negative orthant in $\R^d$, $J$ is the
    projective hyperplane $\sum x_i = 1$, and each $H_i$ is a
    matching hyperplane.

    Define a series of ``working polytopes'' $P_0,P_1,\ldots,P_g$,
    where $P_0 = O \cap J$ and
    $P_i = O \cap J \cap H_1 \cap H_2 \cap \ldots \cap H_i$ for each $i>0$.
    The following inductive algorithm computes the vertices of each
    polytope $P_i$ in turn:

    \begin{enumerate}[1.]
        \item Fill the set $V_0$ with the $d$ unit vectors in $\R^d$.
        Note that $V_0$ is the vertex set for the polytope
        $P_0 = O \cap J$, which is merely the unit simplex in $\R^d$.

        \item \label{en-ddinduct}
        For each $i=1,2,\ldots,g$ in turn, construct a new set $V_i$
        containing the vertices of the polytope $P_i$ as follows:
        \begin{enumerate}[(a)]
            \item Note that $V_{i-1}$ already contains the vertices of
            the previous polytope $P_{i-1}$.  Partition these old vertices
            into three temporary sets $S_0$, $S_+$ and $S_-$,
            containing those $\mathbf{v} \in V_{i-1}$ for which
            $\mathbf{m}^{(i)} \cdot \mathbf{v} = 0$,
            $\mathbf{m}^{(i)} \cdot \mathbf{v} > 0$ and
            $\mathbf{m}^{(i)} \cdot \mathbf{v} < 0$ respectively.
            In other words, $S_0$, $S_+$ and $S_-$ contain those vertices
            in $V_{i-1}$
            that lie in, above and below the hyperplane $H_i$ respectively.
            \item Put the contents of $S_0$ directly into the new vertex
            set $V_i$.
            \item \label{en-ddinduct-pair}
            For each pair $\mathbf{u} \in S_+$ and $\mathbf{w} \in S_-$,
            if $\mathbf{u}$ and $\mathbf{w}$ are adjacent in the
            old polytope $P_{i-1}$ then add the intersection point
            $\overline{\mathbf{u} \mathbf{w}} \cap H_i$ to the new
            vertex set $V_i$.
        \end{enumerate}

        Once steps (a)--(c) are complete, $V_i$ is the vertex set
        for the polytope $P_i$ as required.
        Increment $i$ and proceed to the next iteration of the loop.
    \end{enumerate}

    \noindent % Make it clear that this is outside the enumerated list.
    Upon completion of this algorithm, the vertices of the
    projective solution space $P = P_g$ can be found in the final set $V_g$.
\end{algorithm}

\begin{figure}[htb]
\centerline{\includegraphics[scale=0.85]{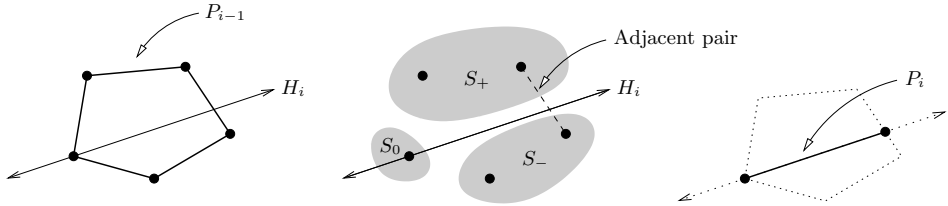}}
\caption{The inductive step of the double description method}
\label{fig-ddinduct}
\end{figure}

The double description method is so named because at each
stage it creates a ``double description'' of the working polytope
$P_i$, both as the intersection $O \cap J \cap H_1 \cap \ldots \cap H_i$
and as the convex hull of the vertex set $V_i$.
The key inductive process of step~\ref{en-ddinduct} is depicted
graphically in Figure~\ref{fig-ddinduct}.

There is one critical detail missing from this algorithm---we need some
way of deciding whether two vertices $\mathbf{u},\mathbf{w} \in V_i$ are
adjacent in the working polytope $P_i$.  There are two primary methods,
one algebraic and one combinatorial.  Both methods are described by
Fukuda and Prodon \cite[Proposition~7]{fukuda96-doubledesc}; we
translate them here into the language of the projective solution space.

\begin{defn}[Zero Set] \label{d-zeroset}
    Consider any point $\mathbf{x} \in \R^d$.  The {\em zero set} of
    $\mathbf{x}$, denoted $Z(\mathbf{x})$, is the set of indices at
    which $\mathbf{x}$ has zero coordinates.  That is,
    $Z(\mathbf{x}) = \{k\,|\,x_k=0\}$.
\end{defn}

Zero sets are important because the non-negative orthant is bounded by
the hyperplanes $x_k=0$.  Thus $Z(\mathbf{x})$ indicates which facets of
the non-negative orthant the point $\mathbf{x}$ belongs to.

\begin{lemma}[Algebraic Adjacency] \label{l-algadj}
    Consider some polytope $P_i$ with vertices
    $\mathbf{u},\mathbf{w}$ in Algorithm~\ref{a-dd}.
    Then $\mathbf{u}$ and $\mathbf{w}$ are adjacent in $P_i$ if and only if
    the intersection of $H_1 \cap \ldots \cap H_i$ with the hyperplanes
    $\{ x_k = 0\,|\,k \in Z(\mathbf{u}) \cap Z(\mathbf{w}) \}$ forms a
    subspace of dimension two.
\end{lemma}

\begin{lemma}[Combinatorial Adjacency] \label{l-combadj}
    Consider some polytope $P_i$ with vertices
    $\mathbf{u},\mathbf{w}$ in Algorithm~\ref{a-dd}.
    Then $\mathbf{u}$ and $\mathbf{w}$ are adjacent in $P_i$ if and only
    if there is no other vertex $\mathbf{z}$ of $P_i$ for which
    $Z(\mathbf{z}) \supseteq Z(\mathbf{u}) \cap Z(\mathbf{w})$.
\end{lemma}

Neither condition is fast to test; the algebraic condition requires the
rank of a matrix, and the combinatorial condition requires yet another
loop through the vertices in $V_i$.  The algebraic test is appealing,
since it is not sensitive to the size of $V_i$ (which can grow very
large).  On the other hand, Fukuda and Prodon report better results with
the combinatorial test, and argue that it should terminate early much
of the time \cite{fukuda96-doubledesc}.  With regards to existing software,
{\fxrays} and {\regina} use the algebraic and combinatorial tests
respectively.

We finish with a final implementation note.  Although we define
the problem in terms of vertices of the projective solution space,
it is easier to work with extremal rays of the polyhedral cone
$N = O \cap H_1 \cap \ldots \cap H_g$, as defined in Definition~\ref{d-proj}.
Abandoning the projective hyperplane allows us to use integer arithmetic
instead of rational arithmetic, which is both faster and easier to implement.
Both {\fxrays} and {\regina} exploit this technique.

\subsection{Filtering for Embedded Surfaces} \label{s-dd-filter}

One critical problem with polytope vertex enumeration is that the number
of vertices can grow extremely large (this is quantified more precisely
in Section~\ref{s-dd-complexity}).  It is therefore in our interests to
avoid generating ``uninteresting'' vertices if at all possible.  The
constraints of normal surface theory allow us to do just this, yielding
spectacular improvements in running time.

Recall from Section~\ref{s-normal} that we are only interested in
vertices that represent {\em embedded} normal surfaces, and that every
such vertex must satisfy the {\em quadrilateral constraints}
(Definition~\ref{d-admissible}).
Each of these constraints identifies three coordinates
$x_i,x_j,x_k$ (representing the three quadrilateral types in some
tetrahedron), and insists that at most one of these coordinates is non-zero.

A na\"ive implementation might generate all vertices of the projective
solution space and then discard those that do not satisfy the quadrilateral
constraints.  However, this does not make vertex enumeration any
faster.  Here we describe a filtering technique that discards such vertices
{\em at every intermediate stage} of the double description method,
thereby reducing the size of each set $V_i$ and speeding up the
subsequent stages of the algorithm.

This filtering method is due to David Letscher, and was used in his
proof-of-concept program {\em Normal} in 1999.  It does not appear in
the current literature, and so we describe it in detail here.

\begin{defn}[Compatibility] \label{d-compatibility}
    Two vectors $\mathbf{u},\mathbf{w} \in \R^d$ are said
    to be {\em compatible} if their sum $\mathbf{u}+\mathbf{w}$ satisfies
    the quadrilateral constraints.
\end{defn}

It is useful to characterize compatibility and the quadrilateral
constraints in terms of zero sets.  The following results are both
immediate consequences of Definitions~\ref{d-admissible}
and~\ref{d-zeroset}:

\begin{lemma} \label{l-zerosetquad}
    A vector $\mathbf{v} \in \R^d$ satisfies the quadrilateral
    constraints if and only if, for each tetrahedron of the underlying
    3-manifold triangulation, $Z(\mathbf{v})$ is missing at most one
    quadrilateral coordinate for that tetrahedron.
\end{lemma}

\begin{lemma} \label{l-zerosetcompat}
    If vectors $\mathbf{u},\mathbf{w} \in \R^d$ contain only non-negative
    elements, then for any $\alpha, \beta > 0$ we have
    $Z(\alpha\mathbf{u} + \beta\mathbf{w}) = Z(\mathbf{u}) \cap Z(\mathbf{w})$.
    In particular, $\mathbf{u}$ and $\mathbf{w}$ are compatible if and
    only if, for each tetrahedron of the underlying 3-manifold
    triangulation, $Z(\mathbf{u}) \cap Z(\mathbf{w})$ is missing at most one
    quadrilateral coordinate for that tetrahedron.
\end{lemma}

It should be observed that all intermediate vertices obtained throughout the
double description method are non-negative, since each intermediate polytope
$P_i$ lies inside the non-negative orthant.
Therefore Lemma~\ref{l-zerosetcompat} can be used in practice
as a fast compatibility test.  We return to this
implementation detail in Section~\ref{s-opt}; in the meantime we proceed
with the main filtering algorithm.

\begin{algorithm}[Vertex Filtering] \label{a-filter}
    Consider the double description method as outlined in Algorithm~\ref{a-dd}.
    Suppose we alter step~2(c), so that a pair $\mathbf{u} \in S_+$,
    $\mathbf{w} \in S_-$ is considered only if vectors
    $\mathbf{u}$ and $\mathbf{w}$ are compatible.

    Then, in the resulting algorithm, each intermediate set $V_i$ will
    contain only those vertices of polytope $P_i$ that satisfy the
    quadrilateral constraints.  In particular, the final set $V_g$ will
    contain only those vertices of the projective solution space
    $P = P_g$ that satisfy the quadrilateral constraints.
\end{algorithm}

This algorithm is mostly easy to implement, though there is one
important difficulty.  Step~2(c) of the double description method
requires us to determine whether two vectors are adjacent in the polytope
$P_{i-1}$; however, since we have filtered out some vertices,
we no longer have access to a complete description of $P_{i-1}$.
Happily this turns out not to be a problem---we return to this adjacency
issue after proving the main algorithm correct.

\begin{proof}
    Algorithm~\ref{a-filter} is simple to prove by induction.
    To avoid confusion, let $V_i^D$ denote the
    vertex sets obtained using the original double description method,
    and let $V_i^F$ denote the new sets obtained using vertex
    filtering.  Our claim is that $V_i^F$ contains precisely those
    vectors $\mathbf{v} \in V_i^D$ that satisfy the quadrilateral constraints.

    To begin, we can observe that $V_i^F \subseteq V_i^D$ for each $i$,
    since filtering cannot create new vertices that were not there originally.
    It suffices therefore
    to consider the fate of each original vertex $\mathbf{v} \in V_i^D$.
    We proceed now with the main induction.

    Our claim is certainly true for $V_0^F = V_0^D$, since these
    initial sets contain only unit vectors.  Suppose the claim is true at
    stage $i-1$, and consider some original vertex
    $\mathbf{v} \in V_i^D$.  There are two possible ways in which the original
    double description method could insert $\mathbf{v}$ into $V_i^D$:
    \begin{enumerate}[(i)]
        \item Vector $\mathbf{v}$ is inserted during step~2(b).
        That is, $\mathbf{v}$ comes from the previous vertex set
        $V_{i-1}^D$ and is found to belong to the matching hyperplane $H_i$.

        Since vertex filtering does not affect step~2(b) of the double
        description method, the filtering algorithm inserts $\mathbf{v}$
        into $V_i^F$ if and only if $\mathbf{v}$ is found in $V_{i-1}^F$.
        By our inductive hypothesis, this is true if and only if $\mathbf{v}$
        satisfies the quadrilateral constraints.

        \item Vector $\mathbf{v}$ is inserted during step~2(c).
        That is, $\mathbf{v}$ does not belong to the previous set
        $V_{i-1}^D$, but is instead created as the intersection
        $\overline{\mathbf{u} \mathbf{w}} \cap H_i$, where
        $\mathbf{u},\mathbf{w} \in V_{i-1}^D$ lie on opposite sides of
        the hyperplane $H_i$.

        We begin by noting that
        $\mathbf{v} = \alpha\mathbf{u} + \beta\mathbf{w}$ for some
        $\alpha,\beta > 0$.  There are two cases to consider:
        \begin{itemize}
            \item If either $\mathbf{u}$ or $\mathbf{w}$ does not satisfy
            the quadrilateral constraints, then the combination
            $\mathbf{v}=\alpha\mathbf{u}+\beta\mathbf{w}$ cannot satisfy
            the quadrilateral constraints.  By our inductive
            hypothesis, the pair $\mathbf{u},\mathbf{w}$ is not found in
            $V_{i-1}^F$ and so $\mathbf{v}$ is (correctly) not added to the
            new set $V_i^F$.
            \item If both $\mathbf{u}$ and $\mathbf{w}$ satisfy the
            quadrilateral constraints then, by Lemmas~\ref{l-zerosetquad}
            and~\ref{l-zerosetcompat}, the new vertex $\mathbf{v}$
            satisfies the quadrilateral constraints if and only if
            $\mathbf{u}$ and $\mathbf{w}$ are compatible.  Here the
            filtering algorithm also acts correctly---the inductive
            hypothesis ensures that both $\mathbf{u},\mathbf{w} \in V_{i-1}^F$,
            and so the filtering algorithm adds $\mathbf{v}$ to $V_i^F$
            if and only if $\mathbf{u}$ and $\mathbf{w}$ are compatible.
        \end{itemize}
    \end{enumerate}
    In each case we find that $\mathbf{v}$ is inserted into $V_i^F$ if
    and only if it satisfies the quadrilateral constraints, and so the
    induction is complete.
\end{proof}

We finish with a discussion of the different adjacency tests.
The algebraic adjacency test (Lemma~\ref{l-algadj}) does not rely on the
vertex set $V_i$, and so we can use it unchanged with the filtering
algorithm.

The combinatorial test (Lemma~\ref{l-combadj}) is more of a
problem, since it requires us to examine every vertex of the
intermediate polytope $P_i$.  This is impossible with the filtering
method, where we deliberately throw away uninteresting vertices of $P_i$
to make the algorithm faster.  Happily this does not matter,
as seen in the following result.

\begin{lemma}[Filtered Combinatorial Adjacency] \label{l-filtercombadj}
    Consider some intermediate polytope $P_i$ in the vertex filtering
    algorithm, and let $V_i$ contain those vertices of $P_i$ that
    satisfy the quadrilateral constraints.
    If vertices $\mathbf{u},\mathbf{w} \in V_i$ are compatible, then
    they are adjacent in the polytope $P_i$
    if and only if there is no other $\mathbf{z} \in V_i$ for which
    $Z(\mathbf{z}) \supseteq Z(\mathbf{u}) \cap Z(\mathbf{w})$.
\end{lemma}

\begin{proof}
Suppose $\mathbf{u}$ and $\mathbf{w}$ are adjacent in $P_i$.  By
Lemma~\ref{l-combadj} there is no vertex $\mathbf{z}$ of $P_i$
for which $Z(\mathbf{z}) \supseteq Z(\mathbf{u}) \cap Z(\mathbf{w})$,
and in particular there is no such $\mathbf{z} \in V_i$.

Alternatively, suppose $\mathbf{u}$ and $\mathbf{w}$ are not adjacent.
By Lemma~\ref{l-combadj} there is some other vertex $\mathbf{z}$ of $P_i$
for which $Z(\mathbf{z}) \supseteq Z(\mathbf{u}) \cap Z(\mathbf{w})$.
Because $\mathbf{u}$ and $\mathbf{w}$ are compatible,
each tetrahedron has at most one quadrilateral coordinate missing from
the set $Z(\mathbf{u}) \cap Z(\mathbf{w})$ (Lemma~\ref{l-zerosetcompat}).
The same thing can therefore be said about the superset $Z(\mathbf{z})$,
and so $\mathbf{z}$ satisfies the quadrilateral constraints
(Lemma~\ref{l-zerosetquad}).  Thus $\mathbf{z} \in V_i$, and the proof
is complete.
\end{proof}

Vertex filtering is essential for any serious implementation of normal
surface enumeration (and is used by all of the implementations discussed
earlier).  The only cost is the new compatibility test in
step~2(c) of the double description method.  On the other hand, vertex
filtering can dramatically reduce the size of the vertex sets $V_i$,
which cuts down both running time
and memory usage as the double description method loops through these
vertex sets.

As a final note, vertex filtering can also be applied to the sister problem
of {\em almost normal surface enumeration}, where we introduce new
``disc'' types such as octagons and tubes.
Embedded almost normal surfaces have additional
rules, such as at most one quadrilateral {\em or} octagonal disc type
per tetrahedron, and at most one octagon or tube in the entire surface.
These rules yield new constraints of the form ``at most one of the following
coordinates may be non-zero'', whereupon similar filtering methods can be
applied.  The reader is referred to \cite{jaco03-0-efficiency} for further
discussion of almost normal surfaces.

\subsection{Worst Cases and Pivoting Algorithms} \label{s-dd-complexity}

Although the double description method is simple and elegant, it is
unfortunately also very slow and very memory-hungry.  We finish this
overview with a discussion of just how bad things can get.

In many ways the double description method is not at fault---the difficulties
are rooted in the problem it aims to solve.  Dyer shows that counting the
vertices of an arbitrary polytope is NP-hard \cite{dyer83-complexity}, and
Khachiyan et al.~show that vertex enumeration over a polyhedron is
NP-hard \cite{khachiyan08-hard}; these results do not bode well.

At the very least, the running time is bounded below by the number of
vertices (i.e., the size of the output), which for bounded polytopes
can grow exponentially large in the polytope dimension.  Specifically,
for a polytope of dimension $p$ with $f$ facets, the upper bound theorem
of McMullen \cite{mcmullen70-ubt} shows the worst case to be
\begin{eqnarray*}
\mbox{number of vertices} &\leq&
    {f - \lceil \frac12 p \rceil \choose \lfloor \frac12 p \rfloor} +
    {f - \lfloor \frac12 p \rfloor - 1 \choose \lceil \frac12 p \rceil - 1}.
\end{eqnarray*}

As an exercise, we can estimate this upper bound in the context
of normal surface enumeration.  Consider a one-vertex triangulation of
a closed 3-manifold containing $n$ tetrahedra.
Extending a result of Kang and Rubinstein \cite{kang04-taut1},
Tillmann \cite{tillmann08-finite}
shows that the matching equations have a solution space of dimension $2n+1$.
Taking the intersection with the projective hyperplane and the
non-negative orthant in $\R^d$, it follows that the projective solution
space is at worst a $2n$-dimensional polytope with $d$ facets.
In the standard framework of Section~\ref{s-normal} where $d=7n$,
McMullen's upper bound becomes $\frac76 {6n \choose n}$; in
Tollefson's quadrilateral space where $d=3n$ this bound becomes
$\frac32 {2n \choose n}$.  Using Stirling's approximation,
these bounds simplify to:
\begin{eqnarray*}
    \mbox{number of vertices in standard space} & \lesssim &
        \frac{7}{2 \sqrt{15 \pi n}} \cdot \left( \frac{6^6}{5^5} \right)^n
        \ \simeq\ \frac{7}{2 \sqrt{15 \pi n}} \cdot 15^n \\
    \mbox{number of vertices in quadrilateral space} & \lesssim &
        \frac{3}{2 \sqrt{\pi n}} \cdot 4^n \\
\end{eqnarray*}
Thus even in quadrilateral space, the number of vertices can potentially grow
as $4^n$.

One critical weakness of the double description method is its memory
consumption---since it involves looping through vertices of the
intermediate polytopes $P_i$, its memory usage is linear in the number
of such vertices (we return to this in detail in Section~\ref{s-opt-product}).
Pivoting methods for polytope vertex enumeration are superior in this
respect.  In particular, Avis and Fukuda describe a pivoting algorithm
that requires virtually no additional memory beyond storage of the input
data \cite{avis92-pivot}; this algorithm is further refined by
Avis \cite{avis00-revised}.

Although pivoting methods are appealing for the general vertex
enumeration problem, they make it difficult to exploit the quadrilateral
constraints.  Pivoting methods essentially map out the vertices of a
polytope by tracing out the simplex algorithm in reverse, moving from
vertex to vertex using different styles of pivot.  The difficulty is
in finding a pivot that can ``avoid'' uninteresting vertices but still map out
the remainder of the polytope.  Indeed, there is no guarantee that the
region of the polytope that satisfies the quadrilateral constraints is
even connected, and in quadrilateral space it is easy to build examples
where this is not the case.

Since even the fastest pivoting method remains bounded by the size of the
output, it is essential that the quadrilateral constraints can be woven
directly into the algorithm---one cannot afford to construct $4^n$
vertices in the worst case if the quadrilateral constraints are able to
make this number orders of magnitude smaller.  For this
reason we focus on the double description method with vertex filtering,
and leave pivoting methods for future work.

\section{Optimizations} \label{s-opt}

The algorithms of Section~\ref{s-dd} describe a ``vanilla'' implementation
of normal surface enumeration, as implemented for instance by older
versions of the software package {\regina}.
In this section we describe a series of optimizations that, as observed
experimentally, yield substantial improvements in both running time and
memory consumption.  The relevant experiments and their results are
summarized in Section~\ref{s-expt}.

The improvements presented here are offered as a guide for researchers
seeking to code up their own implementations.  Section~\ref{s-opt-micro}
begins with a discussion of well-known but important implementation
techniques, including bitmasks and cache optimization.
Section~\ref{s-opt-sort} focuses on ordering the matching hyperplanes
in a way that exploits the structure of the
matching equations and quadrilateral constraints.  In Section~\ref{s-opt-dim}
we extend a technique of Fukuda and Prodon \cite{fukuda96-doubledesc}
that combines features of
both the algebraic and combinatorial adjacency tests.  Finally,
Section~\ref{s-opt-product} presents a technique in which we
store and manipulate only ``essential'' properties of the intermediate
vertices rather than the vertices themselves.

Memory consumption deserves a particular mention here.  As noted in
Section~\ref{s-dd-complexity}, both the running time and memory usage
for the double description method are exponential in the worst case.  Whilst
running time is in theory an unbounded resource (as long as you are patient
enough), memory is not---a typical personal computer has only a couple of
gigabytes of fast memory.  Once this is exhausted (which has happened to
the author many times during normal surface enumeration), the computer
borrows additional ``virtual memory'' from the hard drive.  This virtual
memory is much, much slower, and can have a severe impact not only on
the performance of the algorithm but also on the entire operating system.

It is prudent therefore to give memory consumption just as high a priority
as running time when working on the double description
method.  We address memory indirectly in Section~\ref{s-opt-sort},
and focus on it explicitly with the techniques of
Section~\ref{s-opt-product}.

\subsection{Implementation Techniques} \label{s-opt-micro}

We begin our catalogue of optimizations with some simple implementation tricks.
Though they are well-known, we include them here for reference
because they have been found to improve the running time significantly.

\begin{itemize}
    \item {\em Bitmasks:} \\
    Several components of the double description method work with zero
    sets of vectors.  These include:
    \begin{itemize}
        \item the combinatorial adjacency test (Lemmas~\ref{l-combadj}
        and~\ref{l-filtercombadj}),
        where we only consider a pair $\mathbf{u},\mathbf{w} \in V_i$
        if there is no other $\mathbf{z} \in V_i$ for which
        $Z(\mathbf{z}) \supseteq Z(\mathbf{u}) \cap Z(\mathbf{w})$;
        \item the compatibility test (Lemma~\ref{l-zerosetcompat}),
        where we only consider a pair $\mathbf{u},\mathbf{w} \in V_i$
        if the set $Z(\mathbf{u}) \cap Z(\mathbf{w})$ is missing at most
        one quadrilateral coordinate for each tetrahedron.
    \end{itemize}

    It is therefore convenient to store the zero set alongside each
    vector as we run through the double description algorithm.  This can
    be done with almost no memory overhead using {\em bitmasks}.  For
    instance, in dimension $d \leq 64$, an entire zero set $S$ can be
    stored using a single 64-bit integer (where the $i$th bit is set if
    and only if $i \in S$).  For $d \leq 128$ this can be done using
    two 64-bit integers, and so on.

    Bitmasks are advantageous because they make set operations extremely
    fast---by using bitwise arithmetic on integers, the computer can
    effectively work on all elements of a set in parallel.  For instance, if
    $d \leq 64$ then set intersection can be computing using the
    single {\ccpp} instruction \verb#ans = x & y#, and subset
    relationships can be tested using the single {\ccpp} test
    \verb#if (x & y == y)#.  Computing the size of a set (as required
    by the compatibility test) is a little more complex, but still fast;
    Warren \cite{warren02-hackers-delight} describes several clever methods
    that are far more efficient than looping through and testing each
    individual bit.\footnote{One can avoid this operation entirely by
    replacing each quadrilateral constraint with three ``illegal
    supersets'' of $Z(\mathbf{v})$.  However, this does not scale well
    to almost normal surfaces since the number of illegal supersets is
    quadratic in the size of each constraint.}

    Finally, bitmasks are not only cheap to store but also fast to
    construct.  As the double description algorithm progresses, new
    vectors $\mathbf{v} \in V_i$ are created from old vectors
    $\mathbf{u},\mathbf{w} \in V_{i-1}$ by forming intersections of the
    form $\overline{\mathbf{u}\mathbf{w}} \cap H_i$.
    Lemma~\ref{l-zerosetcompat} shows that
    $Z(\mathbf{v}) = Z(\mathbf{u}) \cap Z(\mathbf{w})$, which can be
    computed using the fast set operations described above.

    It should be noted that the highly streamlined {\fxrays} has used
    bitmasks for compatibility testing for many years (though it
    optimizes their application for some coordinate systems at the expense
    of others).

    \item {\em Cache optimization:} \\
    In his article on optimizing memory access \cite{drepper07-memory},
    Drepper offers programmers advice on how to best utilize the CPU caches.
    One simple rule is that data that is accessed sequentially should be
    stored sequentially; this allows the CPU to prefetch large chunks of data
    from memory and work with the (much faster) caches instead.  To
    illustrate, Drepper makes a na\"ive implementation of the matrix
    product $A \times B$ run over four times as fast simply by storing
    $A$ and $B$ in row major and column major order respectively---this
    works because the data storage follows the sequential order in which
    elements must be accessed to compute the term
    $(A \times B)_{i,j} = \sum_k A_{i,k} B_{k,j}$.

    In the double description method, where the vertex sets $V_i$ can
    grow extremely large, there is a temptation to use techniques that
    avoid large-scale allocations and deallocations of memory.  For instance,
    we might partition vertices into the sets $S_0$, $S_+$ and $S_-$ in-place,
    without allocating additional temporary memory for these sets.
    However, because Algorithm~\ref{a-dd} repeatedly iterates
    through these sets, Drepper's article suggests that we should
    allocate new blocks of memory to store these sets
    sequentially as simple arrays (or, in {\cpp}, contiguous
    \texttt{std::vector} types).  Likewise, we should avoid storing
    vertices in linked list structures---although vector data types
    require occasional large reallocations of memory, they maintain
    sequential data storage where linked lists do not.

    In theory the benefits of sequential
    data access should be well worth the cost of the extra memory
    allocation and deallocation, and the experimental evidence of
    Section~\ref{s-expt} agrees.
\end{itemize}

\subsection{Hyperplane Sorting} \label{s-opt-sort}

It is well known that the performance of the double description method
is highly sensitive to the order in which the hyperplanes are processed.
This is because the ordering of hyperplanes affects the size of each
intermediate vertex set $V_i$, which in turn directly affects both
running time and memory consumption---running time because step~2(c) of
the double description method involves two nested loops over subsets
$S_+,S_- \subseteq V_{i-1}$, and memory consumption because the entire
vertex set $V_i$ must be computed and stored at each stage, ready for
use in the subsequent iteration of the main loop.

Avis et al.~\cite{avis97-howgood-compgeom} present a series of heuristic
options for this ordering, and proceed to manufacture cases in which
each of them performs poorly; Fukuda and Prodon \cite{fukuda96-doubledesc}
also experiment with different heuristic orderings, and obtain best
results with a lexicographic ordering (in which hyperplanes are sorted
lexicographically according to their coefficient vectors).  However,
Avis et al.~highlight the fact that no one heuristic is ``universally
good'', and that any additional knowledge about the problem at hand
should be exploited if this is possible.

In the context of normal surface enumeration, we can exploit the
following facts:
\begin{enumerate}[(i)]
    \item \label{en-sort-sparse}
    Each hyperplane comes from a single matching equation
    of the form $\mathbf{m}^{(i)} \cdot \mathbf{x} = 0$.  These matching
    equations are {\em sparse}---we can see from Definition~\ref{d-admissible}
    that each coefficient vector $\mathbf{m}^{(i)}$ has at most four non-zero
    coordinates.
    \item \label{en-sort-filter}
    The vertex filtering method strips out any
    vertices with ``incompatible'' non-zero quadrilateral coordinates.
    If we can use the matching equations to relate different quadrilateral
    coordinates within the same tetrahedron (in particular, force them
    to be non-zero), we can thereby hope that many vertices will be
    filtered out (thus keeping the vertex sets $V_i$ as small as possible).
\end{enumerate}

We use observation~(\ref{en-sort-filter}) to define a new ordering of
hyperplanes.  Essentially we start with matching equations that only
involve the final few tetrahedra; gradually we incorporate more and more
tetrahedra into our equations until the entire triangulation is covered.
Since the matching equations are sparse, we expect this to be feasible.
The result is as follows:

\begin{algorithm}[Ordering of Matching Hyperplanes] \label{a-sort}
    Consider some hyperplane $H$ in $\R^d$, defined by the matching
    equation $\mathbf{m} \cdot \mathbf{x} = 0$.  We define the
    {\em position vector} $\mathbf{p}(H)$ to be a $(0,1)$-vector of length
    $d$, where the $k$th element of $\mathbf{p}(H)$ is 0 or 1 according
    to whether the $k$th element of $\mathbf{m}$ is zero or non-zero
    respectively.

    We now insert an extra step at the beginning of the double description
    method, which is to reorder the hyperplanes so that
    $\mathbf{p}(H_1) \leq \mathbf{p}(H_2) \leq \ldots \leq \mathbf{p}(H_g)$.
    Here we treat $\leq$ as a lexicographic ordering of position vectors
    (so in dimension $d=3$ for instance, we have
    $(0,1,0) < (0,1,1) < (1,0,0)$ and so on).
\end{algorithm}

\begin{table}[htb]
\centerline{\small \begin{tabular}{c|c|c|c}
Order & Matching equation &
    Coefficient vector $\mathbf{m}^{(i)}$ & Position vector $\mathbf{p}(H_i)$ \\
\hline
1 &         $q_{1,3}=q_{1,2}$         &
    $0,\mgap 0,\mgap 0,\mgap 0,\mgap 0,-1,\mgap 1$ & $0,~0,~0,~0,~0,~1,~1$ \\
2 & $t_{1,2}+q_{1,2}=t_{1,4}+q_{1,1}$ &
    $0,\mgap 1,\mgap 0,-1,-1,\mgap 1,\mgap 0$      & $0,~1,~0,~1,~1,~1,~0$ \\
3 & $t_{1,3}+q_{1,1}=t_{1,2}+q_{1,3}$ &
    $0,-1,\mgap 1,\mgap 0,\mgap 1,\mgap 0,-1$      & $0,~1,~1,~0,~1,~0,~1$ \\
4 & $t_{1,3}+q_{1,3}=t_{1,2}+q_{1,1}$ &
    $0,-1,\mgap 1,\mgap 0,-1,\mgap 0,\mgap 1$      & $0,~1,~1,~0,~1,~0,~1$ \\
5 & $t_{1,1}+q_{1,1}=t_{1,3}+q_{1,2}$ &
    $1,\mgap 0,-1,\mgap 0,\mgap 1,-1,\mgap 0$      & $1,~0,~1,~0,~1,~1,~0$
\end{tabular}}
\caption{Ordering the matching hyperplanes for the Gieseking manifold}
\label{tab-gieseking}
\end{table}

This ordering is illustrated in Table~\ref{tab-gieseking} for the
one-tetrahedron triangulation of the Gieseking manifold; it can be seen
that the position vectors (though not the coefficient vectors) are indeed
sorted lexicographically.  The original matching equations are also
included in the table, using the notation of Definition~\ref{d-admissible}.

In general, the reason we use position vectors is so that
equations involving only the final few tetrahedra will be processed
relatively early, since their position vectors will begin with long
strings of zeroes.  Likewise, equations that involve the
first coordinate of the first tetrahedron will be processed very late
because their position vectors will begin with a one.
We are therefore able to exploit
observation~(\ref{en-sort-filter}) as outlined above.

Like any other hyperplane ordering, Algorithm~\ref{a-sort} is merely a
heuristic.  However, experimentation shows that it performs very well.
This is seen in Section~\ref{s-expt-order}, where we compare it against
several standard heuristics from the literature.

\subsection{Filtering Pairs by Dimension} \label{s-opt-dim}

Recall from Section~\ref{s-dd} that we have two options for testing
whether vertices $\mathbf{u},\mathbf{w} \in V_i$ are adjacent in
the intermediate polytope $P_i$.  These are the algebraic adjacency test
(Lemma~\ref{l-algadj}), and the combinatorial adjacency test
(Lemmas~\ref{l-combadj} and~\ref{l-filtercombadj}).

Fukuda and Prodon compare these tests experimentally, and find in
their examples that the combinatorial test yields better results
\cite{fukuda96-doubledesc}.  However, recall that the combinatorial test
declares vertices $\mathbf{u},\mathbf{w} \in V_i$ adjacent
if and only if there is no other $\mathbf{z} \in V_i$ for which
$Z(\mathbf{z}) \supseteq Z(\mathbf{u}) \cap Z(\mathbf{w})$.  This means that,
in the worst case, the combinatorial test requires looping through the entire
(possibly very large) vertex set $V_i$ in search for such a $\mathbf{z}$.

Fortunately we can break out of this loop early when $\mathbf{u}$ and
$\mathbf{w}$ are non-adjacent (which is expected in the majority of
cases)---we simply exit the loop when such a $\mathbf{z}$ is found.
However, Fukuda and Prodon take this further and identify cases
in which there is no need to loop at all.  Their idea is to use properties
of the {\em algebraic} test that only rely upon {\em combinatorial} data.
Their result, translated into our formulation of the double description
method, is as follows:

\begin{lemma}[Dimensional Filtering] \label{l-dimfilter}
    Consider some intermediate polytope $P_i \subseteq \R^d$ in the double
    description method (Algorithm~\ref{a-dd}), formed as the intersection
    $P_i = O \cap J \cap H_1 \cap H_2 \cap \ldots \cap H_i$.
    If $\mathbf{u}$ and $\mathbf{w}$ are adjacent vertices of $P_i$, then
    \begin{equation} \label{eqn-dimfilter}
    |Z(\mathbf{u}) \cap Z(\mathbf{w})| + i \geq d - 2.
    \end{equation}
\end{lemma}

This is an immediate consequence of the algebraic test (Lemma~\ref{l-algadj}),
which describes the intersection of $|Z(\mathbf{u}) \cap Z(\mathbf{w})| + i$
hyperplanes as a subspace of dimension two.
The real strength of Lemma~\ref{l-dimfilter} is that
it only requires knowledge of the set $Z(\mathbf{u}) \cap Z(\mathbf{w})$.
Therefore we can use it as a fast prefilter for adjacency testing---for
any pair of vertices $\mathbf{u},\mathbf{w} \in V_i$ we first check
(\ref{eqn-dimfilter}), and only run the full combinatorial test
if the inequality holds.

We proceed now to strengthen the original result of Fukuda and Prodon.
Our aim is to replace $i$ with a smaller number in the inequality
(\ref{eqn-dimfilter}), thus filtering out even more non-adjacent pairs
$\mathbf{u},\mathbf{w}$.  A trivial way to do this is to not count redundant
hyperplanes---we can easily change the inequality to
$|Z(\mathbf{u}) \cap Z(\mathbf{w})| + \rank{i} \geq d - 2$, where
$\rank{i}$ is the number of {\em independent} hyperplanes in the
collection $H_1,\ldots,H_i$.

What is perhaps less obvious is that we can also avoid counting any
hyperplane $H_j$ that does not slice {\em between} vertices of the
previous set $V_{j-1}$ (even if this hyperplane is linearly independent
of the others) and that this works even if $V_{j-1}$ is a {\em filtered}
vertex set.  The full result is as follows:

\begin{lemma}[Extended Dimensional Filtering] \label{l-extdimfilter}
    Consider the double description method (Algorithm~\ref{a-dd}), with or
    without vertex filtering (Algorithm~\ref{a-filter}).  Let $H_k$
    be some matching hyperplane, defined by the matching equation
    $\mathbf{m}^{(k)} \cdot \mathbf{x} = 0$.  We say that $H_k$ is
    {\em pseudo-separating} if there exist vertices
    $\mathbf{v}',\mathbf{v}'' \in V_{k-1}$ for which
    $\mathbf{m}^{(k)} \cdot \mathbf{v}' < 0 < \mathbf{m}^{(k)} \cdot \mathbf{v}''$
    (in other words, there are vertices of the old set $V_{k-1}$ on both
    sides of the hyperplane $H_k$).  This definition is
    illustrated in Figure~\ref{fig-pseudosep}.

    Now consider some intermediate polytope $P_i \subseteq \R^d$,
    with two compatible vertices $\mathbf{u},\mathbf{w} \in V_i$.
    If $\mathbf{u}$ and $\mathbf{w}$ are adjacent vertices of $P_i$, then
    \begin{equation} \label{eqn-extdimfilter}
    |Z(\mathbf{u}) \cap Z(\mathbf{w})| + \mathrm{sep}(i) \geq d - 2,
    \end{equation}
    where $\mathrm{sep}(i)$ is the number of pseudo-separating hyperplanes
    in the list $H_1,H_2,\ldots,H_i$.

    \begin{figure}[htb]
    \centerline{\includegraphics{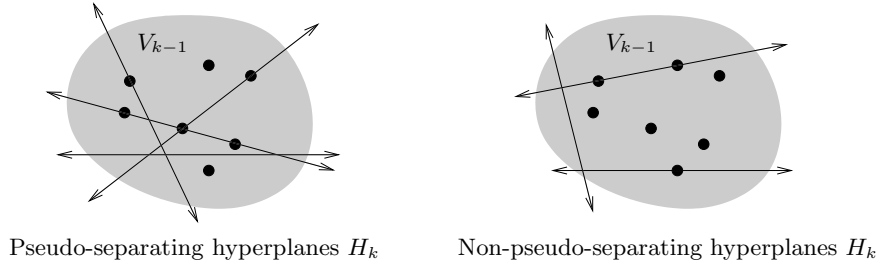}}
    \caption{Illustrating pseudo-separating hyperplanes for a
        vertex set $V_{k-1}$}
    \label{fig-pseudosep}
    \end{figure}
\end{lemma}

Before proving this result, we pause to make some observations.
Not only is this result stronger than Lemma~\ref{l-dimfilter} (since
it is clear that $\mathrm{sep}(i) \leq i$), but it is just as fast to
test---we know when a hyperplane is pseudo-separating because the
corresponding sets $S_+$ and $S_-$ are both non-empty in step~2(a) of
the double description method.

It is again worth noting that Lemma~\ref{l-extdimfilter} remains true
even if we use vertex filtering.  Because each filtered vertex set
$V_i$ is potentially much smaller than the number of vertices of $P_i$,
we can hope to see fewer pseudo-separating hyperplanes as a result
(and thereby strengthen our dimensional filtering).  Indeed, this behaviour
is observed for many ideal triangulations in the cusped census of
Callahan et al.~\cite{callahan99-cuspedcensus}.

We proceed with a proof of Lemma~\ref{l-extdimfilter}.

\begin{proof}
    The following argument assumes the double description method is used
    {\em with} vertex filtering.  For the non-filtered case, simply remove
    all references to filtering in this proof.

    Consider the polytope $P_i = O \cap J \cap H_1 \cap \ldots \cap H_i$
    as described in the statement of Lemma~\ref{l-extdimfilter}, and let
    $V_i$ be the filtered vertex set of $P_i$.
    In the list $H_1,\ldots,H_i$, denote the pseudo-separating hyperplanes by
    $K_1,\ldots,K_p$ and the non-pseudo-separating hyperplanes by
    $L_1,\ldots,L_q$ (so that $p+q=i$).  It is important that
    the hyperplanes in each new list be kept in the same order as in the
    original list $H_1,\ldots,H_i$.

    Define the new polytope $P_i'$ to be the intersection
    $O \cap J \cap K_1 \cap \ldots \cap K_p$ (i.e., the intersection of
    the unit simplex with only the pseudo-separating hyperplanes), and
    let $V_i'$ be the filtered vertex set of $P_i'$.
    Then the original polytope $P_i$ can be expressed as
    $P_i = P_i' \cap L_1 \cap \ldots \cap L_q$.

    We can recover the original filtered vertices $V_i$
    from $V_i'$ using a ``reordered'' double description method.  We begin
    with $P_i'$ and its filtered vertices $V_i'$, and then intersect each
    hyperplane $L_1,\ldots,L_q$ in turn, as described by Algorithms~\ref{a-dd}
    and~\ref{a-filter}.

    In this reordered double description method, let $Q_j$ denote the
    intermediate polytope $Q_j = P_i' \cap L_1 \cap \ldots \cap L_j$, and
    let $W_j$ denote the filtered vertex set of $Q_j$ (in particular,
    $W_0 = V_i'$ and $W_q = V_i$).
    Then we can make the following observations:
    \begin{enumerate}[(i)]
        \item \label{en-ext-newdd}
        Consider the stage of this new double description method where
        we intersect $Q_{j-1}$ with the hyperplane $L_j$ to form the
        polytope $Q_j$.  Then one of the sets $S_+$ and $S_-$ as
        described in step~2(a) of Algorithm~\ref{a-dd} is empty.
        That is, all of the vertices in the previous
        set $W_{j-1}$ lie on $L_j$ and/or to the {\em same side} of $L_j$.

        This can be seen as follows.  Suppose $L_j$ appears as $H_{j'}$
        in the original list $H_1,\ldots,H_i$.  Then
        $Q_{j-1} = H_1 \cap \ldots \cap H_{j'-1} \cap
        K_1 \cap \ldots \cap K_p$ (note that some hyperplanes might be
        repeated in this list).
        Therefore the filtered vertices of $Q_{j-1}$ are all convex
        combinations of the filtered vertices of
        $P_{j'-1} = H_1 \cap \ldots \cap H_{j'-1}$, and since
        $H_{j'}=L_j$ is not pseudo-separating, these vertices
        all lie on and/or to the same side of $L_j$.

        \item \label{en-ext-inc}
        From observation~(\ref{en-ext-newdd}) above,
        the filtered vertex set $W_j$ of the intermediate polytope $Q_j$
        is precisely $W_j = W_{j-1} \cap L_j$.  That is, when we create
        $W_j$ in our reordered double description method, we simply keep
        those vertices of $W_{j-1}$ that lie in the hyperplane $L_j$ and throw
        the others away.  In particular, because one of $S_+$ and $S_-$
        is empty, no {\em new} vertices can be created.

        \item \label{en-ext-supsets}
        Following observation~(\ref{en-ext-inc}) to its conclusion,
        the vertex sets satisfy
        \[ V_i' = W_0 \supseteq W_1 \supseteq \ldots \supseteq W_q = V_i .\]
        As a side note, although $V_i' \supseteq V_i$, it is not necessarily
        true that all {\em unfiltered} vertices of $P_i$ are also vertices of
        $P_i'$.  This is because pseudo-separation is defined only in
        terms of {\em filtered} vertices.
    \end{enumerate}

    Now consider two compatible vertices $\mathbf{u},\mathbf{w} \in V_i$.
    Observation~(\ref{en-ext-supsets}) shows that
    $\mathbf{u}$ and $\mathbf{w}$ are vertices of each polytope $Q_j$.
    Let $F_j$ be the (unique) mini\-mal-di\-men\-sion\-al face of
    $Q_j$ containing both $\mathbf{u}$ and $\mathbf{w}$.  We can
    make the following observations about each face $F_j$:
    \begin{enumerate}[(i)]
        % Continue on from where the previous enumerated list left off.
        \setcounter{enumi}{3}

        \item \label{en-ext-facefilter}
        Every vertex of $F_j$ is in the filtered set $W_j$.

        This can be shown using zero sets.  Let $X \subseteq \R^d$
        be the subspace formed by setting every coordinate in
        $Z(\mathbf{u}) \cap Z(\mathbf{w})$ equal to zero.  That is,
        \[ X = \{\mathbf{x}\,|\,x_i=0\ \mbox{for all}\ i \in
            Z(\mathbf{u}) \cap Z(\mathbf{w})\} .\]
        Because $Q_j$ is a polytope in the non-negative orthant,
        each equation $x_i=0$ is a supporting hyperplane for $Q_j$.
        Therefore $X \cap Q_j$ is a face of $Q_j$, and moreover contains both
        $\mathbf{u}$ and $\mathbf{w}$.  By minimality, $F_j$ is a subface
        of $X \cap Q_j$, and so every vertex of $F_j$ lies in $X$.

        Because $\mathbf{u}$ and $\mathbf{w}$ are compatible,
        Lemma~\ref{l-zerosetcompat} shows that
        $Z(\mathbf{u}) \cap Z(\mathbf{w})$ is missing at most one
        quadrilateral coordinate per tetrahedron.  This means that every
        point in $X$ satisfies the quadrilateral constraints, and in
        particular, so do the vertices of $F_j$.  Thus the vertices of
        $F_j$ (which are also vertices of the enclosing polytope $Q_j$)
        belong to the filtered set $W_j$.

        \item \label{en-ext-nextface}
        For each $j$, the faces $F_{j-1}$ and $F_j$ are identical.

        This follows from our earlier observations.  From
        (\ref{en-ext-facefilter}) and (\ref{en-ext-newdd}) we see that
        $L_j$ is a supporting hyperplane for $F_{j-1}$, and so
        $F_{j-1} \cap L_j$ is a subface of $F_{j-1}$ containing
        $\mathbf{u}$ and $\mathbf{w}$; by minimality it follows that
        $F_{j-1} \cap L_j = F_{j-1}$.

        On the other hand, since $F_{j-1}$ is a face of the polytope
        $Q_{j-1}$, we see that $F_{j-1} \cap L_j$ is a face of
        $Q_{j-1} \cap L_j = Q_j$, again containing both $\mathbf{u}$
        and $\mathbf{w}$.  Thus $F_j$ is a subface of
        $F_{j-1} \cap L_j = F_{j-1}$, and by minimality again it follows
        that $F_j = F_{j-1}$.

        \item \label{en-ext-allfaces}
        Carrying observation~(\ref{en-ext-nextface}) to its conclusion,
        we have $F_0 = F_1 = \ldots = F_q$.
    \end{enumerate}

    We are finally ready to examine the problem of adjacency.
    Vertices $\mathbf{u}$ and $\mathbf{w}$ are adjacent in the polytope
    $Q_j$ if and only if the face $F_j$ is an edge.
    From observation~(\ref{en-ext-allfaces}) it follows that
    $\mathbf{u}$ and $\mathbf{w}$ are adjacent in $P_i'=Q_0$
    if and only if they are adjacent in $P_i=Q_q$.

    Our main result (Lemma~\ref{l-extdimfilter}) now follows immediately
    by applying the earlier Lemma~\ref{l-dimfilter} to the polytope $P_i'$.
\end{proof}

We conclude with a note regarding further generalizations of
Lemma~\ref{l-extdimfilter}.  A key property of non-pseudo-separating
hyperplanes is that, in the double description method, they do not
create any {\em new} vertices (though they can remove old ones).

We might hope therefore to extend Lemma~\ref{l-extdimfilter}
to ``avoid counting'' all hyperplanes with this property; for instance, we
might hope to avoid hyperplanes that are pseudo-separating but that do not
produce any {\em compatible} pairs of vertices on either side.  It turns
out that this cannot be done---although counterexamples are extremely
rare,\footnote{Only one counterexample was found in the entire
10-tetrahedron census of closed orientable and non-orientable
manifolds \cite{burton07-nor10}.} they can be found.

\subsection{Inner Product Representation} \label{s-opt-product}

Recall from the opening notes of Section~\ref{s-opt} that memory (unlike
time) is often a resource with hard limits, and that the worst cases for
the double description method are exponential in memory as well as time.
In our final improvement to the double description method, we take aim
directly at memory consumption.

The high memory usage of the double description method comes almost
entirely from storing the intermediate vertex sets $V_i$.  Traditionally
each vertex is stored as a sequence of coordinates in $\R^d$, though in
Section~\ref{s-opt-micro} we extend this marginally by adding a bitmask
for the zero set.  Therefore, if we are to reduce memory usage, we have
one of two options:
\begin{itemize}
    \item Find a way to reduce the sizes of the vertex sets $V_i$, which
    is the approach taken in Section~\ref{s-opt-sort};
    \item Find a way to avoid storing the full coordinates of each
    vertex, which is the approach that we take here.
\end{itemize}

Our strategy therefore is to discard the individual coordinates of each
vertex $\mathbf{v} \in V_i$, and instead to store only ``essential
information'' from which we can recover the full coordinates if we need
to.  In fact we already have this essential information---it is all
contained in the zero set $Z(\mathbf{v})$.

\begin{lemma} \label{l-zeroenough}
    At any stage of the double description method, the zero set
    $Z(\mathbf{v})$ contains sufficient information to recover the
    entire vertex $\mathbf{v} \in V_i$.  More precisely, the full
    coordinates of $\mathbf{v}$
    can be recovered by solving the following simultaneous equations:
    \begin{itemize}
        \item the matching equations $\mathbf{m}^{(k)} \cdot \mathbf{v} = 0$
        for $k=1,\ldots,i$;
        \item the projective equation $\sum v_j = 1$;
        \item the facet equations $v_j = 0$ for each $j \in Z(\mathbf{v})$.
    \end{itemize}
\end{lemma}

\begin{proof}
    This is an immediate consequence of the fact that a vertex of a
    polytope is the intersection of the facets that it belongs to.
    For each intermediate polytope $P_i$, the facets of $P_i$ are
    described by inequalities of the form $x_j \geq 0$ (deriving from
    the non-negative orthant).\footnote{Some of these inequalities may be
    redundant, describing lower dimensional or empty faces.  Nevertheless,
    every facet is described by an inequality of this form.}
    The facets that a vertex $\mathbf{v}$ belongs to are those for
    which $v_j = 0$, that is, those corresponding to coordinate positions
    $j \in Z(\mathbf{v})$.
\end{proof}

What Lemma~\ref{l-zeroenough} shows is that we could store {\em only the
zero sets} of our vertices, with no coordinates whatsoever.  Because
zero sets are stored as bitmasks that take very little space
(Section~\ref{s-opt-micro}), this would be a magnificent improvement in
memory consumption.  However, it could slow down our algorithm terribly,
since we would need to solve the equations of Lemma~\ref{l-zeroenough}
each time we wanted to analyze or manipulate any vertex $\mathbf{v} \in V_i$.

If we analyze the algorithms and improvements of Sections~\ref{s-dd}
and~\ref{s-opt}, we find that the only operations we need to perform on
vertices are the following:
\begin{enumerate}[(i)]
    \item Creating $d$ unit vectors for the initial set $V_0$;
    \item Testing the sign of $\mathbf{m}^{(i)} \cdot \mathbf{v}$ for
    $\mathbf{v} \in V_{i-1}$;
    \item Creating a new vertex
    $\mathbf{v} = \overline{\mathbf{u}\mathbf{w}} \cap H_i$ from two
    old vertices
    $\mathbf{u},\mathbf{w} \in V_{i-1}$, which is done by computing
    \[ \mathbf{v} = \frac{(\mathbf{m}^{(i)} \cdot \mathbf{u}) \mathbf{w} -
    (\mathbf{m}^{(i)} \cdot \mathbf{w}) \mathbf{u}}
    {(\mathbf{m}^{(i)} \cdot \mathbf{u}) - (\mathbf{m}^{(i)} \cdot \mathbf{w})}\ ;\]
    \item Manipulations involving zero sets (such as compatibility
    testing, the combinatorial adjacency test, or extended dimensional
    filtering);
    \item Outputting the final solutions, as stored in the final vertex
    set $V_g$.
\end{enumerate}
The only non-trivial operation in this list is the inner product
$\mathbf{m}^{(i)} \cdot \mathbf{v}$, which suggests storing vectors using
the following representation:

\begin{defn}[Inner Product Representation] \label{d-inner}
    Consider some vertex $\mathbf{v} \in V_i$ in the double description
    method (Algorithm~\ref{a-dd}).  We define the {\em inner product
    representation} $I(\mathbf{v})$ to be the $(g-i)$-dimensional vector
    \[ \left( \mathbf{m}^{(i+1)} \cdot \mathbf{v},~
         \mathbf{m}^{(i+2)} \cdot \mathbf{v},~
         \ldots,~
         \mathbf{m}^{(g)} \cdot \mathbf{v} \right), \]
    recalling that there are $g$ matching equations in total.

    We use the inner product representation by storing
    $I(\mathbf{v})$ instead of the full coordinates
    for each intermediate vertex $\mathbf{v} \in V_i$.
    We continue to store the zero sets $Z(\mathbf{v})$ as bitmasks so that
    (by Lemma~\ref{l-zeroenough}) no information is lost.
\end{defn}

In general, the inner product representation is cheaper to
store than the full coordinates of a vertex, and grows significantly
cheaper still as the algorithm progresses.  In the standard coordinate
system of Section~\ref{s-normal}, we work in dimension $d=7n$ but have at
most $g \leq 6n$ matching equations.  In Tollefson's quadrilateral
coordinates we work in dimension $d=3n$ but (assuming a closed one-vertex
triangulation) have a mere $n+1$ matching equations.  As a result,
our algorithm starts out using a modest $6/7$ or $1/3$ of the original
storage respectively, and as $i \to g$ in the later stages of the algorithm
our memory consumption shrinks almost to zero.
By the time we reach $i=g$ the only storage remaining is the (very cheap)
bitmasks for our zero sets.

It is important that the greatest benefits of the inner product
representation arise in the later stages of the algorithm.  Anecdotal
evidence (observed time and time again) suggests that the worst explosions
of vertex sets $V_i$ tend to occur in later stages of the algorithm.  This
means that our new representation focuses its optimizations where they are
needed the most.

While the inner product representation gives a significant improvement in
memory consumption, it is important to understand how it affects running
time.  We therefore consider each of the five vertex operations listed earlier:

\begin{enumerate}[(i)]
    \item Creating $d$ unit vectors for the initial set $V_0$ is easy.
    If $\mathbf{v}_j$ is the $j$th unit vector, then the elements of
    $I(\mathbf{v}_j)$ are the $j$th coordinates of the matching equations
    $\mathbf{m}^{(1)},\ldots,\mathbf{m}^{(g)}$.

    \item \label{en-op-sign}
    Testing the sign of $\mathbf{m}^{(i)} \cdot \mathbf{v}$ for
    $\mathbf{v} \in V_{i-1}$ is very easy; we simply look at the first
    element of $I(\mathbf{v})$.

    \item Computing the intersection
    $\mathbf{v} = \overline{\mathbf{u}\mathbf{w}} \cap H_i$ for
    $\mathbf{u},\mathbf{w} \in V_{i-1}$ and $\mathbf{v} \in V_i$
    is much the same as in the standard algorithm.  Using
    \[ \mathbf{v} = \frac{(\mathbf{m}^{(i)} \cdot \mathbf{u}) \mathbf{w} -
    (\mathbf{m}^{(i)} \cdot \mathbf{w}) \mathbf{u}}
    {(\mathbf{m}^{(i)} \cdot \mathbf{u}) - (\mathbf{m}^{(i)} \cdot \mathbf{w})} \]
    it is easy to show that
    \[ I(\mathbf{v}) = \mathrm{trunc}\left[
        \frac{\mathrm{head}[I(\mathbf{u})]\,I(\mathbf{w}) -
        \mathrm{head}[I(\mathbf{w})]\,I(\mathbf{u})}
        {\mathrm{head}[I(\mathbf{u})] - \mathrm{head}[I(\mathbf{w})]}\right],\]
    where $\mathrm{head}[\ldots]$ denotes the first element of a vector,
    and where $\mathrm{trunc}[\ldots]$ indicates removing the first element
    from a vector.

    \item \label{en-op-bits} Manipulations involving zero sets are all done
    using bitmasks, and are not affected at all by the change in vector
    representation.

    \item Outputting the final solutions requires us to solve the
    equations of Lemma~\ref{l-zeroenough} for each vertex in the final
    set $V_g$.
\end{enumerate}

The running times for these operations under both old and new
vertex representations are listed in Table~\ref{tab-repops}
(excluding zero set manipulations, which are irrelevant here).
Since $g < d$ in general, we find that most of these operations
are in fact faster using the inner product representation.

\begin{table}[htb]
\centerline{\small \begin{tabular}{l|l|l}
    Operation & Full coordinate rep. & Inner product rep.\\
    \hline
    Creating $d$ unit vectors & $O(d^2)$ & $O(gd)$ \\
    Testing sign of $\mathbf{m}^{(i)} \cdot \mathbf{v}$ & $O(d)$ & $O(1)$ \\
    Computing $\overline{\mathbf{u}\mathbf{w}} \cap H_i$ & $O(d)$ & $O(g)$ \\
    Outputting final solutions & $O(|V_g|)$ & $O(g^2 d \cdot |V_g|)$
\end{tabular}}
\caption{Time complexities for various vertex operations}
\label{tab-repops}
\end{table}

There is only one operation for which the inner product representation is
slower, which is outputting the final solutions.  Here we must solve a
full system of equations for each vertex in $V_g$ (the complexity estimate in
Table~\ref{tab-repops} assumes a simple implementation using matrix row
operations).  However:
\begin{itemize}
    \item The output operation does not happen often.  We only output
    solutions at the very end of the algorithm, unlike operations
    such as $\mathbf{m}^{(i)} \cdot \mathbf{v}$ or
    $\overline{\mathbf{u}\mathbf{w}} \cap H_i$ which we perform
    many times at every stage of the double description method.

    \item The number of systems of equations we must solve is $|V_g|$,
    which is not large.
    Experimental evidence suggests that the final solution set $V_g$ is
    typically small, often orders of magnitude smaller than the worst
    intermediate vertex sets $V_i$ (see for instance
    Table~\ref{tab-hyp} in the following section).
    This is likely due to the quadrilateral constraints---as we enforce
    more matching equations, we are forced to make more quadrilateral
    coordinates non-zero, and we can filter out more vertices as a result.
\end{itemize}

We hope therefore that this extra cost in outputting the final solutions is
insignificant, and indeed this is seen experimentally in
Section~\ref{s-expt-time}---the losses in the output operation are far
outweighed by the other gains described above, and the inner product
representation yields better performance in both memory usage and running time.

\section{Experimentation} \label{s-expt}

Having developed several improvements to the normal surface
enumeration algorithm, we now road-test these improvements using a collection
of real examples, measuring both running time and memory consumption.

In the following tests, we enumerate surfaces in both the standard
coordinate system of Section~\ref{s-normal} and the quadrilateral
coordinates of Tollefson \cite{tollefson98-quadspace}.
We include both systems because they have some interesting differences:
\begin{itemize}
    \item The matching equations in standard coordinates are {\em all}
    sparse, whereas in quadrilateral coordinates they are only sparse
    {\em on average} (in particular, dense equations are infrequent
    but possible).
    \item The quadrilateral constraints (and hence vertex filtering)
    involve all coordinate positions in
    quadrilateral coordinates, but only $3/7$ of the coordinate positions in
    standard coordinates.
    \item Quadrilateral coordinates work in a smaller dimension than
    standard coordinates ($3n < 7n$), allowing us to run tests on
    larger and more interesting triangulations.
\end{itemize}
For further information on quadrilateral coordinates and the corresponding
matching equations, the reader is again referred to
\cite{tollefson98-quadspace}.

We use 19 different triangulations for our tests: eleven closed hyperbolic
triangulations are used as ``ordinary'' cases, and eight twisted layered
loops are used for extreme ``stress testing''.  In detail:
\begin{itemize}
    \item The closed hyperbolic triangulations are drawn arbitrarily from
    the Hodgson-Weeks census of small-volume closed hyperbolic
    3-manifolds \cite{hodgson94-closedhypcensus}.
    These include six smaller cases ($9 \leq n \leq 13$) for use with standard
    coordinates, and five larger cases ($16 \leq n \leq 20$) for use with
    quadrilateral coordinates.

    \item An $n$-tetrahedron {\em twisted layered loop} is an
    extremely well structured triangulation of the quotient space
    $S^3/Q_{4n}$.  Twisted layered loops are conjectured by Matveev
    to have minimal complexity \cite{matveev98-or6}, and a proof of this claim
    has recently been announced by Jaco, Rubinstein and Tillmann
    \cite{jaco09-coverings}.
    Here we include four smaller cases ($9 \leq n \leq 18$) for
    standard coordinates, and four larger cases ($30 \leq n \leq 75$)
    for quadrilateral coordinates.

    The following properties make twisted layered loops ideal for stress
    testing:
    \begin{itemize}
        \item The tight structure of these triangulations makes vertex
        filtering extremely powerful, allowing us to run tests on very large
        triangulations (up to 75 tetrahedra for quadrilateral coordinates).
        \item In standard coordinates, the final solution set $V_g$ contains
        an exponential number of vertices (specifically
        $F_{n-1} + 2 F_{n-2} + 1$, where $F_0=1$, $F_1=1$, \ldots are the
        Fibonacci numbers).
        Moreover, this is observed to be much larger than the final solution
        set for most census triangulations\footnote{Here we refer to
        censuses of closed compact 3-manifold triangulations, such as
        those described in \cite{burton07-nor10} and
        \cite{hodgson94-closedhypcensus}.}
        of similar size.
        \item In contrast, in quadrilateral coordinates the final solution set
        contains a linear number of vertices (specifically $n+1$), which is
        observed to be very small amongst census triangulations of similar size.
    \end{itemize}
    The reader is referred to \cite{burton09-extreme} for details on the
    final two points, and in particular for proofs of the
    formulae $|V_g| = F_{n-1} + 2 F_{n-2} + 1$ and $|V_g| = n+1$.
\end{itemize}

\begin{table}[htb]
    \centerline{\small \begin{tabular}{r|r|r|r|r}
        \em Tetrahedra & \em Hyp. volume & \em Final set $|V_g|$ &
            \em Max of any $|V_i|$ & \em Dimension $d$ \\
        \hline
        \multicolumn{5}{c}{Standard coordinates} \\
        \hline
        % Apologies for the messy TeX here.
        % The aim is to right-align the numbers but allow the (a) and (b)
        % to stick out to the right.  This requires leaving a (b)-sized
        % space after each ordinary number, and making (a) take up a
        % (b)-sized space in the case of "12(a)".
        9\phantom{(b)}            & 0.94270736 & 19 & 899 & 63 \\
        10\phantom{(b)}           & 1.75712603 & 30 & 873 & 70 \\
        11\phantom{(b)}           & 2.10863613 & 45 & 2\,221 & 77 \\
        12\rlap{(a)}\phantom{(b)} & 2.93565190 & 64 & 3\,477 & 84 \\
        12(b)                     & 3.02631753 & 54 & 941 & 84 \\
        13\phantom{(b)}           & 3.08076667 & 59 & 1\,891 & 91 \\
        \hline
        \multicolumn{5}{c}{Quadrilateral coordinates} \\
        \hline
        16\phantom{(b)} & 4.27796055 & 48 & 6\,655 & 48 \\
        17\phantom{(b)} & 4.30972819 & 33 & 4\,025 & 51 \\
        18\phantom{(b)} & 4.40945629 & 68 & 3\,335 & 54 \\
        19\phantom{(b)} & 4.58232390 & 95 & 15\,988 & 57 \\
        20\phantom{(b)} & 4.68714601 & 156 & 47\,317 & 60
    \end{tabular}}
    \caption{Statistics for the ``ordinary'' closed hyperbolic triangulations}
    \label{tab-hyp}
\end{table}

\begin{table}[htb]
    \centerline{\small \begin{tabular}{r|c|r|r|r}
        \em Tetrahedra & \em Quotient space & \em Final set $|V_g|$ &
            \em Max of any $|V_i|$ & \em Dimension $d$ \\
        \hline
        \multicolumn{5}{c}{Standard coordinates} \\
        \hline
        9 & $S^3/Q_{36}$ & 77 & 375 & 63 \\
        12 & $S^3/Q_{48}$ & 323 & 1\,585 & 84 \\
        15 & $S^3/Q_{60}$ & 1\,365 & 6\,711 & 105 \\
        18 & $S^3/Q_{72}$ & 5\,779 & 28\,425 & 126 \\
        \hline
        \multicolumn{5}{c}{Quadrilateral coordinates} \\
        \hline
        30 & $S^3/Q_{120}$ & 31 & 171 & 90 \\
        45 & $S^3/Q_{180}$ & 46 & 261 & 135 \\
        60 & $S^3/Q_{240}$ & 61 & 351 & 180 \\
        75 & $S^3/Q_{300}$ & 76 & 441 & 225
    \end{tabular}}
    \caption{Statistics for the ``extreme case'' twisted layered loops}
    \label{tab-twisted}
\end{table}

Tables~\ref{tab-hyp} and~\ref{tab-twisted} give an overview of the
19 triangulations chosen for testing; the two tables cover the
hyperbolic triangulations and the twisted layered loops respectively.
The columns in each table include:
\begin{enumerate}[(i)]
    \item The number of tetrahedra $n$.  The hyperbolic set includes
    two 12-tetrahedron triangulations; these are labelled 12(a) and 12(b)
    for later reference.
    \item The hyperbolic volume in Table~\ref{tab-hyp}, and the
    quotient space in Table~\ref{tab-twisted}.  This information, combined
    with the tables from the Hodgson-Weeks census
    \cite{hodgson94-closedhypcensus}, uniquely identifies each 3-manifold.
    \item The size of the final solution set $V_g$.
    \item The maximum size of {\em any} intermediate vertex set $V_i$,
    under an algorithm that uses all of the improvements of
    Section~\ref{s-opt}.
    \item The dimension of the underlying vertex enumeration problem,
    which is $7n$ or $3n$ for standard or quadrilateral coordinates
    respectively.
\end{enumerate}

The maximum $|V_i|$ figures are particularly interesting.  In
Table~\ref{tab-hyp} they highlight the observation that, for ``ordinary''
triangulations, the intermediate sets $V_i$ can grow orders of magnitude
larger than the final set $V_g$.  In Table~\ref{tab-twisted} they highlight
the strength of vertex filtering in the highly structured twisted layered
loops, where the vertex sets $V_i$ are kept small from start to finish.

The remainder of this section is structured as follows.  In
Section~\ref{s-expt-time} we consider the various improvements presented in
this paper and examine their effect on running time for each of our 19
triangulations.  Likewise, Section~\ref{s-expt-mem} offers a similar
analysis of memory consumption.  In Section~\ref{s-expt-order} we evaluate
our heuristic ordering of hyperplanes in more detail, comparing it against
other standard orderings from the literature.  All experiments are conducted
on a 2.4\,GHz Intel Core~2 machine
using the software package {\regina} \cite{regina,burton04-regina}.

\subsection{Improvements in Running Time} \label{s-expt-time}

We begin our series of experiments with an analysis of running time.
Our aim here is to measure the strength of each individual improvement
presented in Section~\ref{s-opt}.

As a starting point, we begin with the standard double description method
with vertex filtering, as described in Algorithms~\ref{a-dd}
and~\ref{a-filter}.  We then refine the algorithm, adding one improvement at
a time, until we arrive at a final algorithm that incorporates all of
the optimizations described in this paper.

\begin{figure}[htb]
    \trialplots{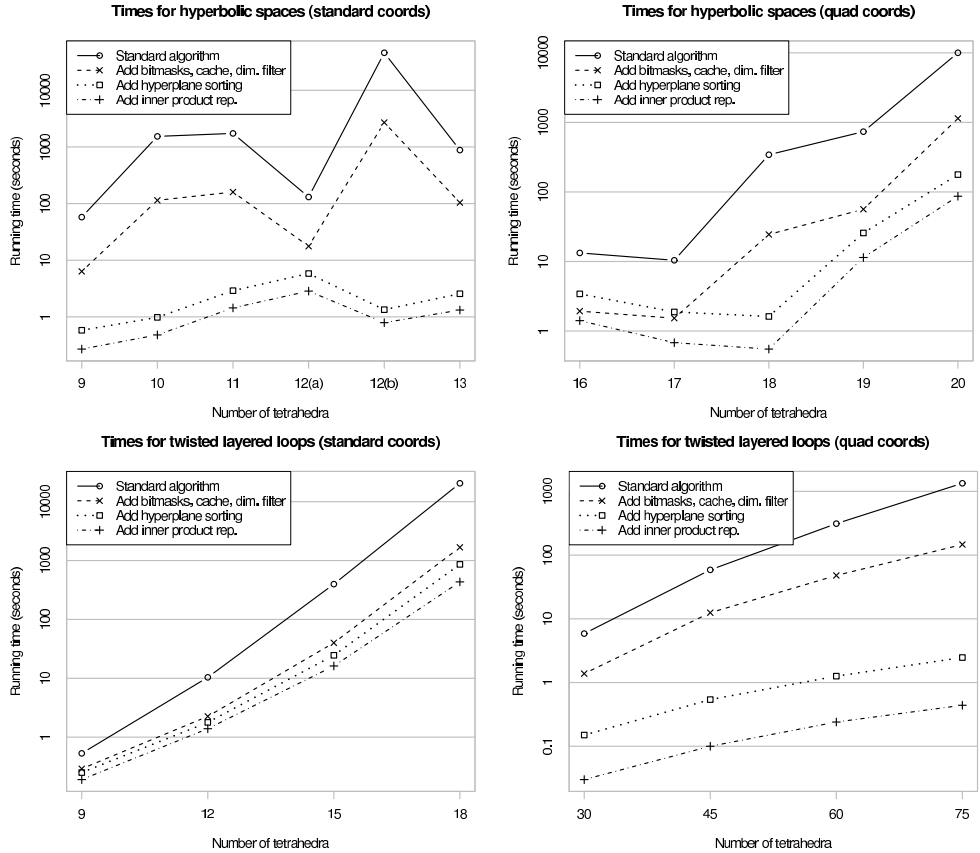}
    \caption{Improvements in running time for major optimizations}
    \label{fig-trial-allinc-time}
\end{figure}

A summary of results is presented in Figure~\ref{fig-trial-allinc-time},
which compares the following variants:
\begin{enumerate}[(i)]
    \item \label{en-times-standard}
    The standard algorithm, as outlined above.
    \item \label{en-times-bits}
    The standard algorithm with bitmasks and cache optimization
    (Section~\ref{s-opt-micro}) and dimensional filtering
    (Section~\ref{s-opt-dim}).  Because each of these optimizations yields
    only a minor improvement on its own, they are bundled together
    to simplify the graphs.
    \item All of the previous improvements plus hyperplane sorting
    (Section~\ref{s-opt-sort}).
    \item \label{en-times-all}
    All of the previous improvements plus the inner product
    representation (Section~\ref{s-opt-product}).
\end{enumerate}

It should be noted that Figure~\ref{fig-trial-allinc-time} is plotted on a
log scale, which means that each horizontal bar represents a
factor of ten improvement.
Given this, the results are extremely pleasing---the final
algorithm~(\ref{en-times-all}) is often 100 or 1\,000 faster than
the original~(\ref{en-times-standard}), and for one case it runs
over 50\,000 times faster.

The weakest improvement is seen with the twisted layered loops using standard
coordinates, where the size of the final solution set is known to be
exponential; here the bitmasks and cache optimization provide most of
the gains.  Nevertheless, even in these extreme scenarios, both
hyperplane sorting and the inner product representation independently double
the speed for the large case $n=18$, and the final algorithm is still
$\simeq 50$ times as fast as the original.

\begin{figure}[htb]
    \trialplots{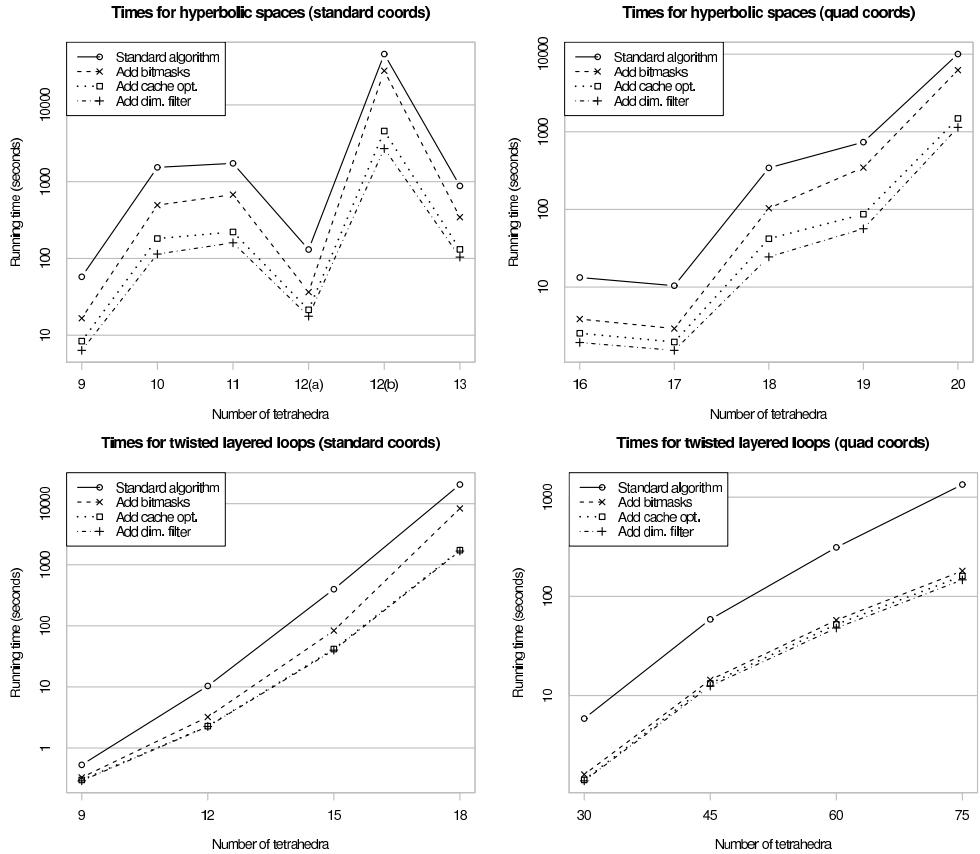}
    \caption{Details for bitmasks, cache optimization and
        dimensional filtering}
    \label{fig-trial-initinc-time}
\end{figure}

Because bitmasks, cache optimization and dimensional filtering are
bundled together in the main summary of results, we separate
them out in Figure~\ref{fig-trial-initinc-time} to show their
individual effects.  Once more we see the extreme nature of the
twisted layered loops---although all three improvements are
effective on the hyperbolic spaces, some improvements (particularly
dimensional filtering) have very little effect on the twisted layered loops.

It is worth noting that in the best cases, such as where the final algorithm
is $> 1\,000$ times or even $> 10\,000$ times faster, the bulk of the
gains are due to hyperplane sorting.  We return to hyperplane
sorting in greater detail in Section~\ref{s-expt-order}.

\subsection{Improvements in Memory Usage} \label{s-expt-mem}

We continue our series of experiments by measuring the memory consumption
of different variants of our algorithm.  The results are plotted in
Figure~\ref{fig-trial-allinc-mem}, where again we bundle together bitmasks,
cache optimization and dimensional filtering for simplicity.

\begin{figure}[htb]
    \trialplots{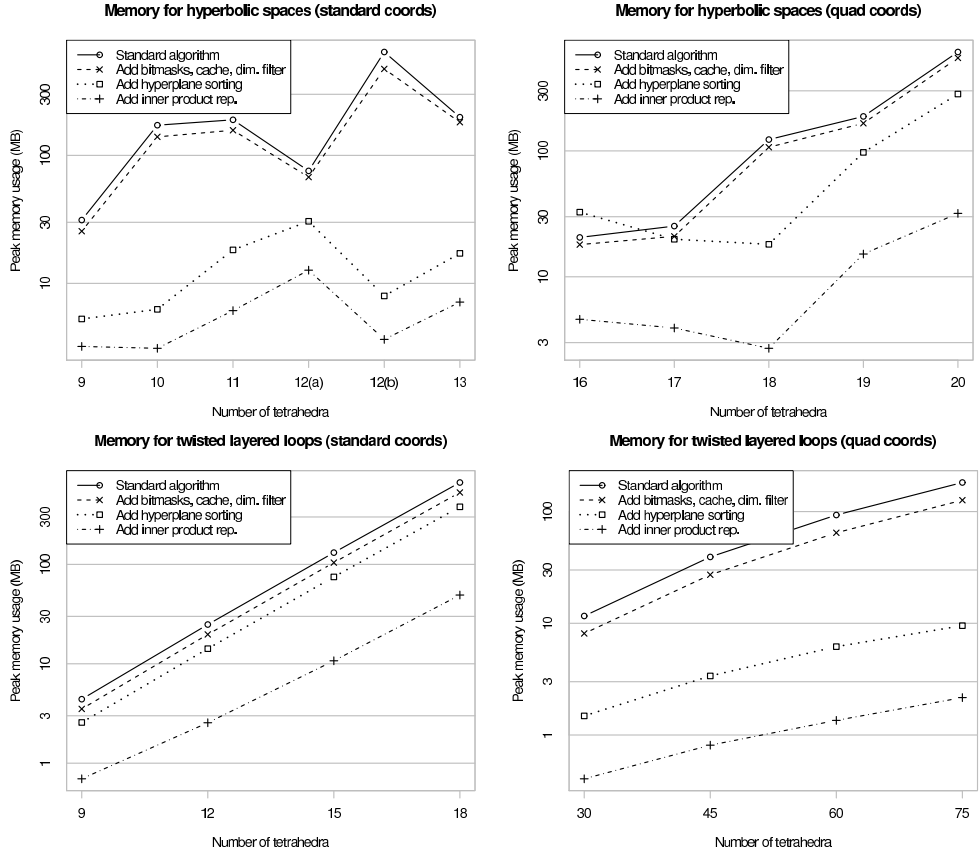}
    \caption{Improvements in memory usage for major optimizations}
    \label{fig-trial-allinc-mem}
\end{figure}

To be precise, Figure~\ref{fig-trial-allinc-mem} measures {\em peak memory
usage}, which is defined to be the maximum memory usage at any stage of the
algorithm {\em minus} the memory usage at the beginning of the algorithm.
This means that we only count memory that is genuinely used by the vertex
enumeration algorithm, and not unrelated overhead such as
system libraries or the storage of the program itself.  We measure memory in
megabytes, which we take to mean $10^6$ bytes (not $2^{20}$ bytes
which is sometimes used instead).

Once more, the results are plotted on a log scale; here every
two horizontal bars represent a factor of ten improvement (with a single bar
representing approximately a factor of three).  Again the results are
extremely pleasing---for the large cases we are able to reduce memory
consumption by factors of around 15 to 85, and in one hyperbolic case by
a factor of 175.

It is particularly interesting to examine memory consumption for the
twisted layered loops.  These cases are extreme in both senses---in
standard coordinates we see our weakest improvements (a factor
of $14$ for $n=18$), and in quadrilateral coordinates we see some of
our strongest improvements (a factor of $84$ for $n=75$).  This is not entirely
surprising, since we know that twisted layered loops have extremely large
and extremely small solution sets in standard and
quadrilateral coordinates respectively.

\subsection{Comparison of Hyperplane Orderings} \label{s-expt-order}

It is noted by Fukuda and Prodon \cite{fukuda96-doubledesc} that the ordering
of hyperplanes is critically important for a fast implementation of the
double description method.  This matches our experimental
observations---whereas most optimizations are consistent in the way they
reduce running time and memory consumption, hyperplane ordering is
much more variable.  Sometimes we only achieve mild improvements through
ordering the hyperplanes, and other times we achieve spectacular results.

The reason that hyperplane ordering has such power is because, unlike any of
our other optimizations, it affects the {\em sizes} of the intermediate
vertex sets $V_i$.  By keeping these sets small and taming the exponential
explosion, we can achieve magnificent improvements in both running time and
memory usage; on the other hand, if we inadvertently encourage the
exponential explosion then the results can be disastrous.

It is therefore prudent to compare our ordering by position vectors
(Algorithm~\ref{a-sort}) against other standard orderings from the literature.
The other orderings we consider are:
\begin{itemize}
    \item {\em No ordering:} \\
    We do not order the hyperplanes at all, but merely process them in the
    order in which they are constructed.  Note that this is not a ``random''
    ordering; in the case of {\regina}, the hyperplanes are constructed in
    an order that (in a rough sense) moves the non-zero coefficients
    from the first tetrahedron to the last.  This is because each matching
    equation involves a face of the triangulation, and {\regina} happens
    to number faces internally in a similar manner.

    \item {\em Dynamic ordering:} \\
    Here we reorder the hyperplanes on the fly.  Recall from
    Algorithm~\ref{a-dd} that each hyperplane $H_i$ is used to divide
    the vertices of $V_{i-1}$ into sets $S_0$, $S_+$ and $S_-$, whereupon we
    embark upon the slow task of examining all pairs
    $\mathbf{u} \in S_+$ and $\mathbf{v} \in S_-$.
    With a dynamic ordering, we choose the hyperplane $H_i$ so that the
    number of pairs $|S_+| \times |S_-|$ is as small as possible.

    This is essentially the dynamic {\em mixcutoff} ordering defined by
    Avis et al.~\cite{avis97-howgood-compgeom}, adapted to make better
    use of the set $S_0$ (whose vertices do not need
    processing).\footnote{Strictly speaking,
    {\em mixcutoff} chooses the hyperplane that makes $S_+$ and $S_-$ the
    most unbalanced.}  Other dynamic orderings appear in the literature,
    notably {\em mincutoff} and {\em maxcutoff}
    \cite{avis97-howgood-compgeom,fukuda96-doubledesc}, but these are defined
    for intersections of half-spaces and are less relevant for intersections
    of hyperplanes.

    \item {\em Lexicographic ordering:} \\
    With lexicographic ordering we simply sort the hyperplanes by their
    coefficient vectors, possibly after performing some normalization.
    Fukuda and Prodon report good results using this method
    \cite{fukuda96-doubledesc}.

    Lexicographic orderings are typically defined for intersections of
    half-spaces, where the sign of each vector is well-defined.  Since we
    are dealing with intersections of hyperplanes, sign does not matter
    (so the coefficient vector $-\mathbf{m}$ is just as good as $\mathbf{m}$).

    We consider two ways of choosing the sign of each vector:
    \begin{itemize}
        \item {\em Positive first}, where we ensure that the first non-zero
        entry in each coefficient vector is positive;
        \item {\em Random signs}, where the sign of each vector is selected
        at random.
    \end{itemize}
\end{itemize}

\begin{figure}[htb]
    \trialplots{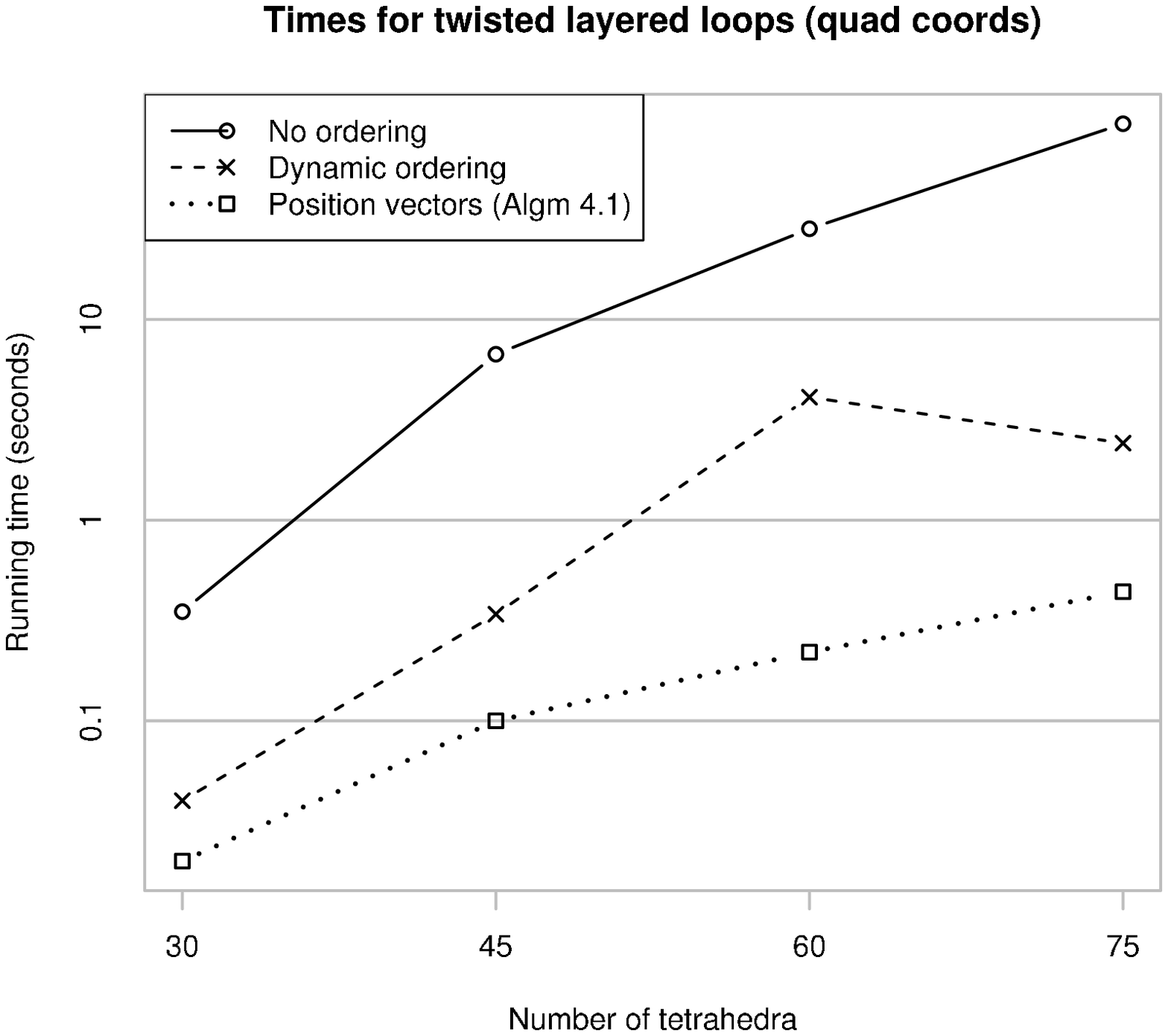}
    \caption{Running times for various hyperplane orderings (set 1/2)}
    \label{fig-trial-order-orig-time}
\end{figure}

\begin{figure}[htb]
    \trialplots{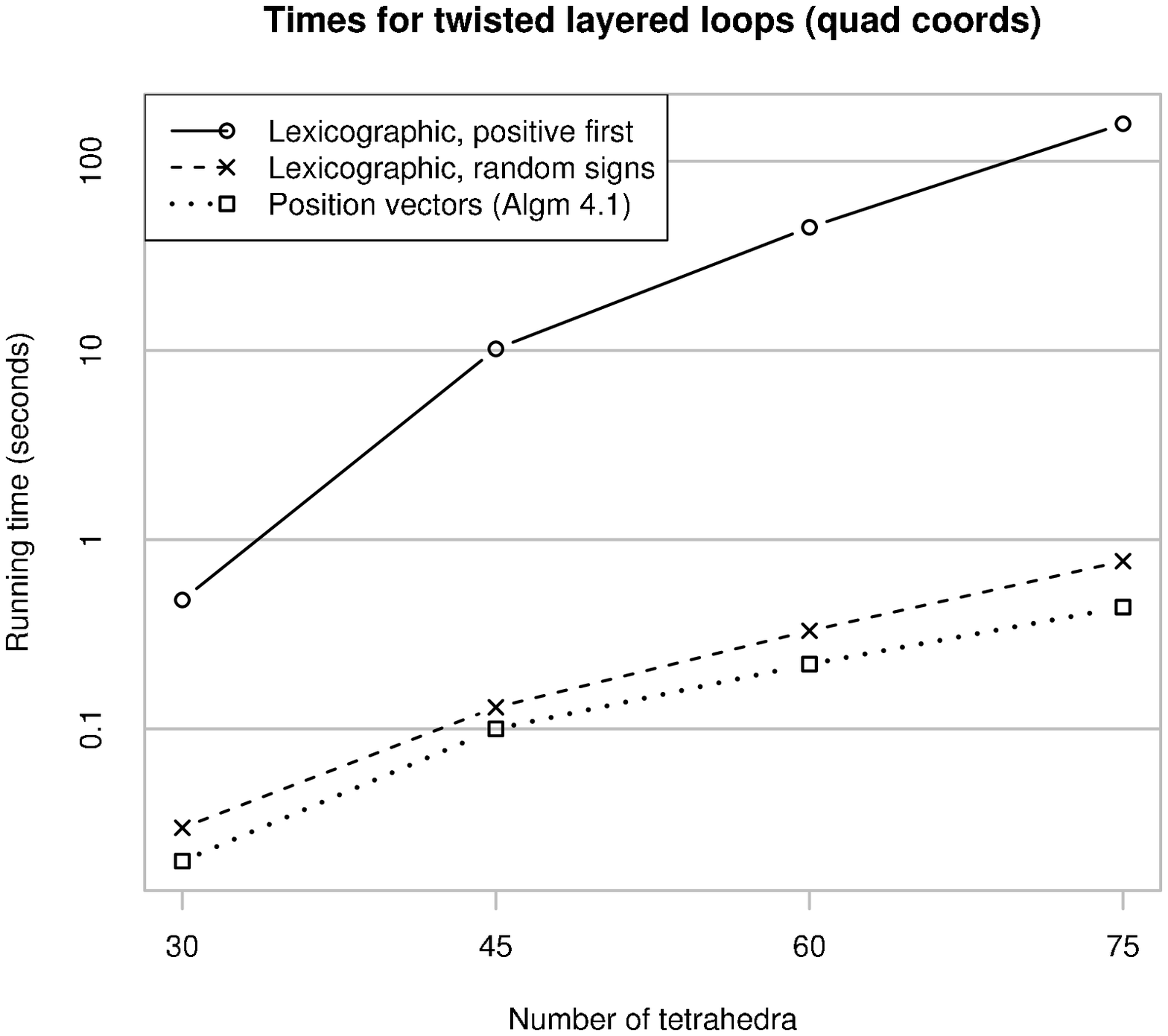}
    \caption{Running times for various hyperplane orderings (set 2/2)}
    \label{fig-trial-order-lex-time}
\end{figure}

The running times for various lexicographic orderings are presented in
Figures~\ref{fig-trial-order-orig-time} and~\ref{fig-trial-order-lex-time}.
Figure~\ref{fig-trial-order-orig-time} compares our Algorithm~\ref{a-sort}
against no ordering and dynamic ordering, and
Figure~\ref{fig-trial-order-lex-time} compares Algorithm~\ref{a-sort}
against both variants of the lexicographic ordering.
Both figures again use a log scale (with each horizontal bar representing a
factor of ten), and the algorithms incorporate all of our other improvements
(bitmasks, cache optimization, dimensional filtering and inner product
representation).

It is pleasing to see that our Algorithm~\ref{a-sort} performs better than
the others in most cases, and is the only ordering that performs
{\em consistently} well.  The only serious competitor is the dynamic
ordering, which performs a little better in some cases; however, for some of
the hyperbolic spaces the dynamic ordering runs $10\,000$ times slower.

As a final note, Figure~\ref{fig-trial-order-lex-time} is missing a
data point.  This is because the {\em random first} lexicographic ordering
for the $n=18$ twisted layered loop was stopped manually after two days;
extrapolation suggests that it could have run for weeks before finishing.

\section{Conclusion} \label{s-conc}

In this paper we outline the standard algorithm for enumerating normal
surfaces in a 3-manifold triangulation, by combining the double
description method of Motzkin et al.~with the vertex filtering method
of Letscher.  Following this we describe four optimizations:
\begin{itemize}
    \item {\em Bitmasks and cache optimization}, which are well-known
    implementation techniques that can be applied to the double description
    method;
    \item {\em Hyperplane sorting}, where we order the matching hyperplanes
    according to their position vectors;
    \item {\em Dimensional filtering}, where we extend a result of Fukuda
    and Prodon to avoid processing certain pairs of vertices;
    \item {\em The inner product representation}, where we store only
    essential properties of the vertices instead of the full vertex
    coordinates.
\end{itemize}

We find that all of these techniques are successful in reducing running time,
with dimensional filtering the weakest (though still effective in most
cases) and hyperplane sorting the strongest (sometimes cutting running time
by several orders of magnitude).  The optimizations that focus on memory
are also successful in reducing memory consumption by significant factors
(though not as large as running time).  Furthermore, the hyperplane ordering
that we define here performs consistently well against other orderings from
the literature.

Whilst these results are extremely promising, readers are encouraged to try
these techniques for themselves---as other authors have noted, the
performance of the double description method is highly variable,
and different examples can reward or penalize different optimizations
\cite{avis97-howgood-compgeom,fukuda96-doubledesc}.  Nevertheless, the
techniques presented here are found to perform consistently well, and are
offered as a basis for further optimizations.

\bibliographystyle{amsplain}
\bibliography{pure}

\vspace{1cm}
\noindent
Benjamin A.~Burton \\
Department of Mathematics, SMGS, RMIT University \\
GPO Box 2476V, Melbourne, VIC 3001, Australia \\
(bab@debian.org)

\end{document}